\documentclass{article}

\usepackage[english]{babel}

\usepackage[letterpaper,top=2cm,bottom=2cm,left=3cm,right=3cm,marginparwidth=1.75cm]{geometry}

\usepackage{amsmath}
\usepackage{amsthm}
\usepackage{graphicx}
\usepackage[colorlinks=true, allcolors=blue]{hyperref}
\usepackage{epsfig,epsf,fancybox}
\usepackage{algorithm}
\usepackage{algorithmicx}
\usepackage{amsfonts,amssymb}
\newtheorem{assumption}{Assumption}
\newtheorem{theorem}{Theorem}[section]

\newtheorem{lemma}[theorem]{Lemma}

\DeclareMathOperator*{\argmin}{arg\,min}

\newcommand{\Var}{\mathrm{Var}}
\newcommand{\email}[1]{\texttt{#1}}

\title{Novel Optimization Techniques for Parameter Estimation}

\author{
Chenyu Wu\thanks{Department of Industrial \& Systems Engineering, University of Minnesota
(\email{wu000766@umn.edu}).}
\and 
Nuozhou Wang\thanks{Department of Industrial \& Systems Engineering, University of Minnesota
(\email{wang9886@umn.edu}).} 
\and 
Casey Garner\thanks{School of Mathematics, University of Minnesota   (\email{garne214@umn.edu}).}
\and 
Kevin Leder\thanks{Department of Industrial \& Systems Engineering, University of Minnesota
(\email{lede0024@umn.edu}).} 
\and Shuzhong Zhang\thanks{Department of Industrial \& Systems Engineering, University of Minnesota
(\email{zhangs@umn.edu}).}
}

\begin{document}
\maketitle

\begin{abstract}
In this paper, we introduce a new optimization algorithm that is well suited for solving parameter estimation problems. 
We call our new method cubic regularized Newton with affine scaling (CRNAS). 
In contrast to so-called first-order methods which rely solely on the gradient of the objective function, our method utilizes the Hessian of the objective. 
As a result it is able to focus on 
points satisfying the {\it second-order}\/ optimality conditions, 
as opposed to first-order methods that simply converge to critical points. 
This is an important feature in parameter estimation problems where the objective function is often non-convex and as a result there can be many critical points making it is near impossible to identify the global minimum. 
An important feature of parameter estimation in mathematical models of biological systems is that the parameters are constrained by either physical constraints or prior knowledge. 
We use an affine scaling approach to handle a wide class of constraints. 
We establish that CRNAS identifies a point satisfying $\epsilon$-approximate second-order optimality conditions within $O(\epsilon^{-3/2})$ iterations. 
Finally, we compare CRNAS with MATLAB's optimization solver \textit{fmincon} on three different test problems. 
These test problems all feature mixtures of heterogeneous populations, a problem setting that CRNAS is particularly well-suited for. 
Our numerical simulations show CRNAS has favorable performance, performing comparable if not better than \textit{fmincon} in accuracy and computational cost for most of our examples. 
\end{abstract}

\section{Introduction}

    

    

Mathematical modeling plays a critical role in understanding and quantifying biological systems. In particular, it enables the quantification of dynamics, thus providing the ability to predict potential outcomes and guide future studies. An essential component of successful mathematical modeling is parameter estimation.
Here the goal is to train and/or fit a given mathematical model to observed data. 
The most common method of parameter estimation follows the maximum likelihood framework where one formulates a statistical model with associated parameters
and maximizes a likelihood function over the parameter space to fit the model to provided data.  
Thus, parameter estimation problems often reduce to optimization problems. 

A common feature of mathematical models arising in the study of biological systems is a high nonlinear dependence on the models' parameters. 
For example, the parameter sensitivities in a system of nonlinear differential equations also satisfy a nonlinear system of differential equations \cite{cao2003adjoint}. 
Even when one has an explicit formulation of their mathematical model there is often a nonlinear dependence on the model parameters, e.g., the logistic growth model, 
$y(t)=L/(1+e^{-k(t-t_0)})$, has a nonlinear dependence on the parameters $k$ and $t_0$.

Parameters in mathematical models for biological systems 
usually have 
a physical interpretation.
So, the values they can assume are often restricted based on our knowledge of the natural world. 
For example, 
non-negativity and upper-bound constraints are often desired if not necessary. 
This is in contrast to neural network models where the parameters are often unconstrained. 
The presence of constraints on the parameter space adds non-trivial difficulties to the derived parameter estimation optimization problem. 
For instance, for constrained problems one cannot simply apply gradient descent since the updates could take the parameters out of their feasible region. 

Another phenomena to consider when modeling biological systems is heterogeneous populations. 
For example, when studying how cells respond \textit{in vitro} to chemotherapy there can be a heterogeneous response \cite{daigeler2008heterogeneous}. 
To model this type of heterogeneity one often assumes there are a known number of distinct subpopulations $S$ present at unknown proportions $\{p_i\}_{i=1}^S$, each with their own distinct characteristics. 
The presence of these unknown proportions can introduce a new challenge to the parameter estimation problem, namely equality constraints. 
In particular, we know \textit{a priori} any set of unknown proportions satisfy the conditions: $\sum_{i=1}^Sp_i=1$ and $p_i > 0$ for $i=1,\hdots,S$.


Combining all of this together, we see parameter estimation problems for biological models utilizing a maximum likelihood framework reduce to nonlinear and non-convex constrained optimization problems
\cite{wu2024using,ljung2013convexity,moles2003parameter}. 
The loss of convexity is an immense challenge. Convexity in optimization enables us to infer global properties from local properties, i.e., local minimums are also global minimums.  
This inferring up from local to global properties is lost with the loss of convexity. Global optimal solutions are preferred, of course, and this is no different in parameter estimation. 
Multiple global optimization frameworks have been suggested to address this issue \cite{schmiester2021efficient,gabor2015robust}. 
These frameworks try to propose a strategy for locating global optimum, while relying on solutions provided by local optimization algorithms which converge to locally optimal solutions.
A frequently employed technique is a multi-start approach; 
this technique suggests to solve the optimization problem starting from randomly selected initial points and obtaining multiple local solutions. 
Among these local solutions, one can argue the best local solution is a global solution if the problem has only a finite number of local solutions; 
however, confirming a finite number of local solutions becomes challenging when 
dealing with nonlinear and non-convex constrained continuous optimization models. 
In summary, current global optimization frameworks rely heavily on locating many local solutions to 
the constrained optimization problem, emphasizing the need for efficient and accurate local methods. 

Many of the likelihood functions utilized in the study of mathematical biology have accessible higher-order derivative information; however, numerous optimization algorithms do not take advantage of this; procedures such as gradient descent and quasi-Newton methods only use the gradient information of the objective function, which categorizes them as first-order methods.
For this reason, these approaches typically only yield solutions satisfying the first-order optimality conditions, which might not be locally optimal. 
In the parameter estimation problems we consider it is possible to compute the Hessian of the objective function and utilize this information to obtain better solutions. 
So,
in this work, we explore the advantages of optimization algorithms utilizing second-order information and develop a new algorithm. 
Our proposed method is called {\it cubic regularized Newton based on affine scaling} (CRNAS). 
This method combines the ideas of the cubic regularized Newton's method and affine scaling in-order to provide a novel second-order scheme capable of solving constrained optimization problems. 
An important advantage of CRNAS over first-order methods is it finds solutions satisfying the second-order optimality conditions. 
Due to the high levels of nonlinearity 
in mathematical biology models, there are possibly many solutions satisfying the first-order optimality conditions which are far from optimal;
therefore, bypassing these solutions is crucial and converging to second-order stationary points has been shown to produce globally optimal solutions for some non-convex models \cite{boumal2016non,ge2016matrix}.  
We provide a convergence analysis for CRNAS, obtaining the best possible iteration complexity bound for the class of models we study, 
and numerical experiments showcase a simple implementation of CRNAS is already competitive with MATLAB's state-of-the-art solvers.

The paper is structured as follows: Section \ref{sec: Methodology} develops CRNAS for a general class of constrained non-convex problems. 
We discuss the two methodologies grafted together to construct CRNAS, the cubic regularized Newton's method
and the affine scaling method,
and present the convergence theory; 
Section~\ref{sec: Benchmark algorithms} summarizes the state-of-the-art algorithms we compared CRNAS with in our numerical experiments and details the metrics used to compare their respective efficiencies; 
Sections~\ref{sec: Case study} and \ref{sec: case study: mixed logistic growth} present numerical experimentation with CRNAS using two case studies;
Section~\ref{sec: Case study} investigates a cancer drug response estimation problem while Section~\ref{sec: case study: mixed logistic growth} studies 
a general application in heterogeneous logistic growth estimation; 
the paper concludes in Section \ref{sec: Conclusion} with a discussion of the advantages and limitations of CRNAS and a proposed implementing scenario; 
additional appendices present the technical arguments of our analysis, exposition on CRNAS's implementation, and details about the set-up of our numerical experiments.

\section{Methodology}\label{sec: Methodology}
\subsection{Optimization Framework}
In a parameter estimation problem, one has a function that maps inputs and a parameter set to a desired output. 
Then, for a given dataset one is interested in finding the best fitting (in some sense) parameter set. 
How one defines `best fitting' is a choice that the modeler makes based on their understanding of the dataset and its generation. 
In particular, assume we have a function $f$ that takes inputs $x$ and parameter set $\theta$ and generates a prediction of a desired output, $y=f(x;\theta)$. If we have a dataset of paired inputs and outputs, $(x_i,y_i)_{i=1}^n$, and assume a simple statistical model generating the outputs, $y_i=f(x_i;\theta)+Z_i$, where $Z_i$ are independent and identically distributed random variables with probability density function $\phi$ that are used to represent potential observation noise,
we can then write the negative log-likelihood of the parameter set given the observed data as
\[
L\left(\theta|(x_i,y_i)_{i=1}^n\right)=-\sum_{i=1}^n\log \phi\left(y_i-f(x_i;\theta)\right).
\]
To complete the parameter estimation problem, we must minimize the function $L\left(\theta|(x_i,y_i)_{i=1}^n\right)$ over $\theta$; however, as mentioned before, the components of $\theta$ might have a physical interpretation and therefore be constrained. For example, they might be constrained to be non-negative, or bounded, or, as stated in the introduction, they might be proportions that sum to one. 
As a result, we have reduced the parameter estimation problem to a constrained optimization problem. 
Note that when writing the negative log-likelihood we will often drop the explicit dependence on the observed data.

Our focus in this work is creating algorithms to numerically solve a broad class of constrained optimization problems that arise from parameter estimation problems of the form
\begin{equation}\label{eq:basic_model} 
    \begin{array}{ll}
\min_{\theta \in \mathbb{R}^n} & L(\theta) \\
\text { s.t. } & A\theta=b,  l\leq \theta\leq u,
\end{array}
\end{equation}
where the objective function $L:\mathbb{R}^n \rightarrow \mathbb{R}$ is possibly non-convex and the domain is the intersection of a linear and box constraint.
Observe, all the examples given for constraints on the parameters can be represented by the conditions in \eqref{eq:basic_model}, 
and we can actually further generalize our optimization framework by noting the box constraint, $l\leq \theta \leq u$, 
can be rewritten as a conic constraint. Letting $\theta_1:=\theta-l$ and $\theta_2:=u-\theta$
we see the condition
$\theta_1,\;\theta_2\geq 0$ and $\theta_1+\theta_2=u-l$ is equivalent to the box constraint. 
So, we can replace the box constraint with an additional linear equality constraint and restrictions to the non-negative orient, i.e., $\mathbb{R}^n_+ := \{ x = (x_1, \hdots, x_n) \; |\: x_i \geq 0,\; i=1,\hdots, n\}$. 

The set $\mathbb{R}^n_+$ is an example of a pointed convex cone, that is, a set $\mathcal{K} \subseteq \mathbb{R}^n$ such that for all $x, y \in \mathcal{K}$ and $t \geq 0$: $tx \in \mathcal{K}$, $x+y \in \mathcal{K}$, and $-\mathcal{K} \cap \mathcal{K} = \{0\}$. If additionally the span of $\mathcal{K}$ is the entire ambient space we say the cone is solid.
Thus, due to this equivalence, we shall instead develop our algorithm to solve the more general class of minimization problems 
%
%
%
%
\begin{equation} \label{main:problem}
    \begin{array}{ll}
\min _{\theta \in \mathbb{R}^n} & L(\theta) \\
\text { s.t. } & A\theta=b, \theta\in \mathcal{K},
\end{array}
\end{equation}
where $\mathcal{K} \subseteq \mathbb{R}^n$ is a pointed convex solid cone and $L: \mathcal{K} \rightarrow \mathbb{R}$ is smooth but possibly non-convex. 
For a first-order method, one typically expects to find a first-order solution, which is defined as $\nabla L(\theta) = 0$ without constraints. 
However, like the examples in Sections \ref{sec: Case study} and \ref{sec: case study: mixed logistic growth}, functions may have many saddle points which satisfy first-order conditions while failing to be local minimums, and 
second-order methods are ideal to avoid converging to such saddle points. 
In the optimization literature, one of the major recent developments is the introduction of second-order methods which converge to second-order stationary points \cite{curtis2019inexact,o2021log,royer2020newton}, 
but relatively few papers in the literature deal with developing second-order methods for non-convex constrained models. 
Independent of our work, in a very recent arXiv paper, Dvurechensky and Staudigl \cite{dvurechensky2024barrier} proposed a similar approach to ours to solve non-convex constrained optimization models. 
The main difference between our work and theirs lies in the fact that we use the Dikin ellipsoid and a cubic regularization of the objective in defining our subproblems, while they use only a cubic regularization of the objective to form their subproblems.

In this paper, we use the definitions for first-order (FOSP) and second-order stationary points (SOSP) as introduced in \cite{he2023newton}. 
The authors of \cite{he2023newton} also introduced a so-called Newton-conjugate gradient (Newton-CG) based barrier method 
to find an $(\epsilon, \sqrt{\epsilon})$-SOSP for non-convex conically constrained models with $O(\epsilon^{-3/2})$  Cholesky factorizations and $\tilde{O}(\epsilon^{-3/2}\min\{n,\epsilon^{-1/4}\})$ other fundamental operations. 
In comparison, the authors of \cite{dvurechensky2024barrier} apply a cubic regularized model to get a $O(\epsilon^{-3/2})$ worst-case iteration bound.
CRNAS also obtains a $O(\epsilon^{-3/2})$ worst-case iteration bound to $\epsilon$-SOSP for \eqref{main:problem}; 
CRNAS is simple to implement and is well suited for the situation when the Hessian is easily obtainable. Additionally, our approach avoids the need to solve a possibly ill-conditioned Newton's equation; we only require the solution to a simple subproblem which we discuss in Appendix \ref{subproblem}. 
Next, we introduce the notion of self-concordant barrier functions for conic constraints, which play an important role in our analysis, and the corresponding notion of 
$\epsilon$-FOSP and $\epsilon$-SOSP arising from them.  

\subsection{Logarithmic Homogeneous Self-concordant Barrier Functions} \label{optimality}
Logarithmic homogeneous self-concordant barrier functions \cite{nesterov1994interior} play an important role in interior-point methods.
In this paper, we assume the cone $\mathcal{K}$ is convex pointed and solid and has a logarithmic homogeneous self-concordant barrier function $B : {\rm int}(\mathcal{K}) \to \mathbb{R}$, that is, $B$ is convex, three-times continuously differentiable over ${\rm int}(\mathcal{K})$, $B(\theta_k) \rightarrow \infty$ for all sequences $\{\theta_k\}_{k \in \mathbb{N}} \subseteq {\rm int}(\mathcal{K})$ which converge to a point on the boundary of $\mathcal{K}$, and for any $\theta\in {\rm int}(\mathcal{K})$, $\tau > 0$, and direction $h\in \mathcal{R}^n$, the following properties are satisfied:
\[|\nabla^3 B(\theta)[h,h,h]|\leq 2(\nabla^2 B(\theta)[h,h])^{3/2},\]
\[B(\tau\theta)=B(\theta)- D\ln{\tau},\]
where $D$ is a positive constant.
As an example, $B(\theta)=-\sum_{i=1}^n \log(\theta_i)$ is a logarithmic homogeneous self-concordant barrier function for the cone $\mathbb{R}^n_{+}$ with $D=1$. 

An important aspect of barrier functions comes from the fact they can define induced local norms.
These norms are standard in the conic optimization literature and play a crucial role in our algorithm development and analysis. 
Following convention,  we let
\begin{eqnarray}
    \|v\|_\theta&:=&\left(v^\top \nabla^2 B(\theta) v\right)^{1 / 2} \quad \nonumber \\ 
    \|v\|_\theta^*&:=&\left(v^\top \left[\nabla^2 B(\theta)\right]^{-1} v\right)^{1 / 2} \nonumber \\
    \|C\|_\theta^* &:=&\max _{\|v\|_\theta \leq 1}\|C v\|_\theta^* \nonumber 
\end{eqnarray}
for all $v \in \mathbb{R}^n$ and $C \in \mathbb{R}^{n \times n}$.
Note, unless otherwise stated, $\|\cdot\|$ denotes the regular Euclidean norm. 
Following the pattern of \cite{he2023newton}, a feasible solution $\hat \theta \in \mbox{\rm int}(\mathcal{K})$ with $A\hat \theta=b$ is said to be an $\epsilon$-approximate first-order stationary point ($\epsilon$-FOSP) 
for \eqref{main:problem} if 
\[
\mbox{\rm dist}\left((\nabla^2 B(\hat\theta))^{-1/2} \nabla L(\hat\theta), (\nabla^2 B(\hat\theta))^{-1/2}\mbox{\rm Range}(A^\top)\right) = O(\epsilon),
\]
which means that the angle between $(\nabla^2 B(\hat\theta))^{-1/2} \nabla L(\hat\theta)$ and the orthogonal complement subspace of $(\nabla^2 B(\hat\theta))^{-1/2}\mbox{\rm Range}(A^\top)$, namely $ (\nabla^2 B(\hat\theta))^{1/2}\mbox{\rm Null}(A)$, is of order $\epsilon$. The latter statement can be restated as in terms of the local norms as
\begin{equation} \label{KKT1}
|(\nabla L(\hat \theta))^\top  d|\leq \epsilon \|d\|_{\hat \theta}
\end{equation}
for all $Ad=0$. Moreover, a solution 
$\hat \theta$ is called an
$\epsilon$-approximate second-order stationary point ($\epsilon$-SOSP)
if in addition to being an $\epsilon$-FOSP the point $\hat\theta$ also satisfies
\[
d^\top \nabla^2 L(\hat\theta) d\geq  - \sqrt{\epsilon} 
\|d\|^2_{\hat\theta} 
\]
for all $A d = 0$; see Remark 1 (ii) of \cite{he2023newton} for further details. Next, we introduce the two methods which inspired our algorithm design. 


\subsection{Cubic Regularized Newton’s Method}
The cubic regularized Newton’s method was first proposed by Nesterov and Polyak to solve smooth unconstrained optimization problems \cite{nesterov2006cubic}. 
Their approach adds a cubic regularization term to the second-order Taylor expansion of the objective function about the current iterate, $\theta^k$, and solves this subproblem to compute the next iterate, i.e., 
\begin{equation}\label{eq:crn_subproblem}
\theta^{k+1} \in \arg\min_{\theta \in \mathbb{R}^n} \left( \left\langle\nabla L\left(\theta^k\right), \theta-\theta^k\right\rangle +\frac{1}{2}\nabla^2L\left(\theta^k\right)[\theta-\theta^k]^2+\frac{M}{6}\left\|\theta-\theta^k\right\|^3\right).
\end{equation}
If the objective function has a gradient Lipschitz Hessian with constant $L_H \geq 0$ and $M \geq L_H$, then the above subproblem amounts to minimizing a cubic upper bound of the objective function about $\theta^k$. Hence, under mild assumptions, cubic regularized Newton monotonically decreases the value of the objective function.  

This method boasts multiple desirable qualities. 
The algorithm converges globally to second-order stationary points and computes a point satisfying the approximate first-order unconstrained optimally conditions, i.e., $\|\nabla L(\theta)\|\leq \epsilon$, within $O(\epsilon^{-3/2})$ iterations, which bests gradient descent, and the method converges quadratically near strict local minimums \cite{nesterov2006cubic}. Furthermore, though \eqref{eq:crn_subproblem} is generally a non-convex problem, the subproblem is actually equivalent to minimizing a convex function in one variable. 
Thus, cubic regularized Newton is an implementable, globally convergent, and locally fast converging method to second-order stationary points.  
For these reasons, we seek to incorporate aspects of this procedure into our design.  

\subsection{Affine Scaling Method}
The affine scaling method \cite{lagarias1990ii,barnes1986variation,vanderbei1986modification} is an algorithm for solving linear programming problems by rescaling 
the variables 
and constraining the next iterate to lie within a ball contained inside the cone $\mathbb{R}^n_{+}$ to get a better iterate at a reduced computational cost. 
Consider the linear program
\[\begin{array}{ll}
\min _{\theta \in \mathbb{R}^n} & c^\top \theta \\
\text { s.t. } & A\theta=b, \theta\geq {0}.
\end{array}\]
In each iteration, we scale the problem based on the current point $\theta^k>0$. 
Let $\theta'=D_k^{-1}(\theta-\theta^k)=D_k^{-1}\theta-\textbf{1}_{n}$, where $D_k=\mbox{\rm Diag}(\theta^k)$ is a diagonal matrix with diagonal elements $\theta^k$ and $\textbf{1}_{n}$ is the vector of all ones in $\mathbb{R}^n$. 
Then, the non-negativity constraint is equivalent to $\theta'\geq -\textbf{1}_{n}$. 
We desire to have the next iterate stay inside the interior of $\mathbb{R}^n_+$, so we replace $\theta'\geq -\textbf{1}_{n}$ with a ball constraint, that is, 
we instead enforce $\|\theta'\|\leq 1-\alpha$, where $0<\alpha<1$. 
This then ensures the next iterate cannot lie on the boundary of the cone with the proximity to the boundary dictated by $\alpha$, i.e., $\alpha$ close to zero can yield new iterates near the boundary.  
With this new constraint and change-of-variable, we form 
the following subproblem
\begin{equation}\label{eq:affine_scaling_subprob}
\begin{array}{ll}
\min _{\theta \in \mathbb{R}^n} & (D_kc)^\top \theta' \\
\text { s.t. } & AD_k\theta'={0}, \|\theta'\|\leq 1-\alpha.
\end{array}
\end{equation}
The main benefit of \eqref{eq:affine_scaling_subprob} is it has a closed formed solution. 
So, the affine scaling method proceeds by forming and solving \eqref{eq:affine_scaling_subprob} and using the variable transformation to obtain the next iterate. 
Utilizing the notation introduced in Section~\ref{optimality}, we can then concisely write the affine scaling update as
\begin{equation}\label{eq:affine_iterate}
\begin{array}{llll}
\theta^{k+1}&=&\arg\min _{\theta \in 
          \mathbb{R}^n} & c^\top \theta \\
            & & \text { s.t. } & A\theta=b, \|\theta-\theta^k\|_{\theta^k}\leq 1-\alpha
\end{array}
\end{equation}
where $B(\theta) = -\sum_{i=1}^{n} \log(\theta_i)$ is the barrier function used to define the norm in the constraint.

\subsection{Cubic Regularized Newton with Affine Scaling (CRNAS)}
Our new method seeks to combine the ideas of cubic regularized Newton and affine scaling to develop a procedure which solves \eqref{main:problem}. 
The crux of our approach is to utilize the affine scaling method to handle the constraints, leveraging the concept of homogeneous self-concordant barrier functions to deal with the general conic constraints, but replace the linear objective in \eqref{eq:affine_iterate} with the local cubic approximation of the objective function in \eqref{eq:crn_subproblem}. 
Thus, our method generates new iterates by solving the following subproblem
    \begin{eqnarray}\label{eq:main_subproblem}
        \theta^{k+1}&=&\arg\min _{A\theta=b,\|\theta-\theta^k\|_{\theta^k}\leq 1-\alpha} \left(\left\langle\nabla L\left(\theta^k\right), \theta-\theta^k\right\rangle +\frac{1}{2}\nabla^2L\left(\theta^k\right)[\theta-\theta^k]^2+\frac{M}{6}\left\|\theta-\theta^k\right\|_{\theta^k}^3\right),
    \end{eqnarray}
where $M$ is a positive number. Since our method brings together these two different procedures we coined it cubic regularized Newton with affine scaling (CRNAS). 
A precise description of CRNAS is provided below. 
\vspace{0.05in}

\shadowbox{\begin{minipage}{6.4in}
{\bf Cubic Regularized Newton with Affine Scaling (CRNAS)}\label{alg}
\begin{description}
\item[Step 0:] Provide an interior point $\theta^0$, i.e., $A\theta^0=b$ and $\theta^0\in \mbox{\rm int}(\mathcal{K})$; choose the constants $\eta >0$, $M >0$, and $\alpha \in (0,1)$; set $k=0$

\item[Step 1:] Solve the following subproblem and proceed to Step 2: 
\begin{eqnarray}
        \theta^{k+1}&=&\arg \min _{\theta:A\theta=b,\|\theta-\theta^k\|_{\theta^k}\leq 1-\alpha} \left(\left\langle\nabla L\left(\theta^k\right), \theta-\theta^k\right\rangle +\frac{1}{2}\nabla^2L\left(\theta^k\right)[\theta-\theta^k]^2+\frac{M}{6}\left\|\theta-\theta^k\right\|_{\theta^k}^3\right)\nonumber
    \end{eqnarray}
\item[Step 2:] If $\|\theta^{k+1}-\theta^k\|_{\theta^k}<\eta$, let $K=k+1$ and stop. Otherwise, go back to Step 1 with $k=k+1$
\end{description}
\end{minipage}}
\vspace{0.01in}

Though \eqref{eq:main_subproblem} appears to be more difficult to solve than \eqref{eq:crn_subproblem}, this subproblem forming the backbone of CRNAS is solvable.
A theoretical analysis and practical approach to solving \eqref{eq:main_subproblem} is detailed in Appendix~\ref{subproblem}. 
In the next section, we provide an overview of the convergence guarantees we have for CRNAS; the technical proofs of the results are left to Appendix~\ref{complex}. 

\subsection{Overview of the Complexity Analysis of CRNAS} 

We first declare CRNAS is well-defined in the sense all of the iterates produced by the algorithm remain inside the cone $\mathcal{K}$.
The following lemma guarantees this follows from the constraints in the subproblem.
\begin{lemma}\label{lemma:feas} (Theorem 2.1.1, \cite{nesterov1994interior})
    Let $B$ be self-concordant on $\mathcal{K}$ and let $\theta_0\in  \mbox{\rm int}(\mathcal{K})$, then $\{\theta:\|\theta-\theta_0\|_{\theta_0}<1\}\subset  \mbox{\rm int}(\mathcal{K})$.
\end{lemma}

For our convergence theory to hold, we assume the following scaled Lipschitz smoothness of the second-order derivative of the objective: 
\begin{assumption}\label{lip} 
There exists a constant $\beta \geq 0$ such that 
for all $x, y \in \mbox{\rm int}(\mathcal{K})$
$$\|\nabla^2 L(y)-\nabla^2 L(x)\|_x^*\leq \beta \|y-x\|_x.$$
\end{assumption}
\noindent We note this assumption is not original to us but has precedence in the  
literature, e.g.~\cite{he2023newton}. 
So, under these limited conditions, we present
our main complexity theorem; the proof can be found in Appendix \ref{complex}.
\begin{theorem} \label{thm:iteration_complexity}
    If Assumption~\ref{lip} holds and CRNAS is applied to \eqref{main:problem} with
   %
    $M=2\beta$ and $\eta=\min\{1-\alpha,\epsilon^{-1/2}\alpha^{1/2}M^{-1/2},\frac{1}{\sqrt{2}}\epsilon^{-1/2}\alpha^{2}M^{-1}\}$, then the algorithm will stop within $K\leq 12(L(\theta^0)-L^*)\eta^{-3}+1$ iterations with $\theta^K$ an $\epsilon$-SOSP.
\end{theorem}



Therefore, in view of the discussions on the optimality conditions in Subsection~\ref{optimality}, Theorem~\ref{thm:iteration_complexity} implies that in at most $O(\epsilon^{-3/2})$ iterations CRNAS is guaranteed to find an $\epsilon$-SOSP. So, for the highly nonlinear and non-convex optimization problems 
coming from parameter estimation problems, CRNAS is able to avoid the many sub-optimal first-order stationary points leading to a proposed solution of often superior quality compared to first-order methods. 

It is worth remarking one can devise a first-order version of CRNAS which uses a quadratic approximation of the objective function in the subproblem rather than a cubic approximation. 
A similar convergence result can be derived for this procedure as well. 
As situations arise where second-order derivatives are prohibitively expensive to compute, this version of CRNAS would be prudent to implement; therefore, for the sake of completeness, we include a description of our first-order version of CRNAS with associated convergence theory in Appendix~\ref{first-complexity}
\section{Benchmarking Algorithms and Performance Metrics}
\label{sec: Benchmark algorithms}

In our case studies, we analyzed, evaluated, and compared the performance of CRNAS to state-of-the-art constrained nonlinear programming algorithms in MATLAB's optimization solver \textit{fmincon} \cite{MatlabOTB}. 
Specifically, we compared CRNAS with the \textit{interior-point} and \textit{sequential quadratic programming} algorithms.
Although \textit{fmincon} implements other approaches, our experience indicated these two methods performed the best in our experiments, so we only consider these in our exposition. 

Since CRNAS relies on precise gradient and Hessian information, we considered implementations of algorithms in \textit{fmincon} both with and without specifying the gradient and Hessian of the objective function.
We implemented both the \textit{interior-point} and \textit{sequential quadratic programming} algorithms in {\it fmincon} by supplying only the objective function and allowing MATLAB's built-in techniques to estimate the gradient and Hessian; IP and SQP refer to these implementations.
We also tested the {\it interior-point} solver with the exact gradient and Hessian; IP hess denotes this instantiation.
Lastly, the {\it sequential quadratic programming} method with the exact gradient was employed in our testing; SQP grad denotes this procedure.
Table~\ref{tab: Algorithm} summarizes these different approaches.

\begin{table}[ht]
    \centering
    \begin{tabular}{l|l|l}
    \hline
        {\bf Algorithm} & {\bf Gradient/Hessian} & {\bf Acronym} \\
    \hline
        Cubic regularized Newton with affine scaling & Yes & CRNAS\\
    \hline
        \textit{Interior-point} in \textit{fmincon} & No  & IP \\
    \hline
        \textit{Interior-point} in \textit{fmincon} & Yes & IP hess\\
    \hline
        \textit{Sequential quadratic programming} in \textit{fmincon} & No & SQP \\
    \hline
        \textit{Sequential quadratic programming} in \textit{fmincon} & Yes & SQP grad\\
    \hline
        
    \end{tabular}
    \caption{Summary of the algorithms tested in our numerical experiments}
    \label{tab: Algorithm}
\end{table}

Our studies sought to analyze the effectiveness and quality of the solutions provided by each algorithm. 
We evaluated three quantities in our assessment of each method: 
\begin{enumerate}
\item{\bf Total Compute Time} - the total wall time required for Algorithm $x$ to run from $N$ different initial points and terminate. Termination will occur if a local solution is obtained up to some provided tolerance or some computational resource or numerical limit has been expended or reached, e.g., maximum iterations or minimum stepsize, respectively

\item{\bf Best Computed Value} - is the lowest value of the objective function obtained from Algorithm $x$ from $N$ different initializations; this value is Algorithm $x$'s best estimate of the global minimum after $N$ different instantiations

\item{\bf Number of Iterations to Best Value} - is the number of iterations required to obtain the best computed value from the single initial point which generated it   
\end{enumerate}

The first two metrics seek to provide a measure of which method shall be of greatest benefit to a practitioner.
These metrics measure which algorithm is most effective at producing the best solution the quickest when a multi-start approach is applied to estimate a global minimum.
The third metric on the other hand provides an estimate of total iterations needed for the method to compute its best solution from a single initialization.
To ensure a fair comparison, in all of our tests each algorithm was provided the same initial guesses and had the same termination criteria. 
The termination criteria used in the numerical experiments are detailed in Appendix \ref{appx: stopping criteria}.
In the next two sections, we set-up parameter estimation problems arising in mathematical biology and compare CRNAS against {\it fmincon} using these three measures.


\section{Case Study I: Cancer Drug Response Estimation}\label{sec: Case study}
In this section, we explore the use of CRNAS to solve a recently proposed parameter estimation problem  \cite{kohn2023phenotypic,wu2024using}.
In these works, the authors employed a maximum likelihood estimation (MLE) approach to estimate tumor subpopulation dynamics and the corresponding drug response from high throughput drug screening data. 
The primary challenge in these studies arises from estimating the parameters of mixtures of dose-response Hill equations which are both nonlinear and non-convex. 
To alleviate some of the issues arising from nonlinearity, we perform a parameter transformation on the variables in the Hill equations.
Due to the importance and prevalence of Hill equations, we review them before describing the cancer drug response estimation problem.  

\subsection{Hill Equations and our Parameter Transformation}

The Hill equation is a fundamental model for describing dose-response relationships in pharmacological studies \cite{gesztelyi2012hill}. For a given drug dose level $d$, the Hill equation specifies the fraction of viable cells as
$$
    H(d;b,E,n) = b + \frac{1-b}{1 + (\frac{d}{E})^n},
$$
where the parameters $b$, $E$, and $n$ represent the maximum drug effect, the half maximal effect dose (EC50), and the Hill coefficient respectively. 
This model has been well utilized in pharmacology, biochemistry, and various other fields over the past hundred years \cite{hill1910possible}. 

An important question that arises in the study of the Hill equation is the accurate identification of model parameters based on observed data \cite{gadagkar2015computational}. 
A particularly challenging aspect of dealing with the parameters of the Hill equation is the nonlinear term in the denominator, $E^n$. 
To mitigate the challenge of this term, we apply a variable transformation.
In particular, we introduce a new parameter $\mathcal{E} = E^{n}$ to produce the modified Hill equation 
$$
    \Bar{H}(d;b,\mathcal{E},n) = b + \frac{1-b}{1+\frac{d^n}{\mathcal{E}}}.
$$
Given the transformation $\mathcal{E} = E^n$ is a one-to-one function when $E$ is non-negative, we can determine the value of $E$ and $n$ from the estimated values of $\mathcal{E}$ and $n$. 
This parameter transformation is applied during the numerical experiments, while the original EC50 parameter $E$ is used throughout the manuscript for clarity.

\subsection{Deterministic Drug-Affected Cell Proliferation Framework}\label{sec: Deterministic drug-affected framework}

\subsubsection{Problem Description}\label{sec: Deterministic HTS problem description}
In \cite{kohn2023phenotypic}, the authors proposed a novel statistical framework, PhenoPop, to analyze high-throughput drug screening data. 
PhenoPop takes into account the tumor drug response when multiple subpopulations, each with varying drug responses, exist within the tumor. 
A single Hill equation cannot capture this heterogeneous drug response; 
therefore, the PhenoPop model represents the dynamics of tumor growth as a composite of numerous subpopulations, with each subpopulation distinguished by a unique set of Hill parameters. 

PhenoPop employs a classical exponential growth model to describe the growth of each subpopulation. Each subpopulation has a distinct growth rate, denoted as $\alpha_i$. The size of subpopulation $i$ at time $t$ is represented as
\begin{equation*}
    \label{eq:exponential growth}
    X_i(t) = X_i(0) \exp(\alpha_i t),
\end{equation*}
where $X_i(0)$ is the initial population size of subpopulation $i$. 
The unique set of Hill parameters for subpopulation $i$ is denoted as $(b_i,E_i,n_i)$. The relationship between the cell growth rate and the drug dose is expressed as
\[\alpha_i(d; b_i,E_i,n_i) = \alpha_i + \log(H(d;b_i,E_i,n_i)),\]
where the constraint $b_i \in (0,1)$ is enforced to ensure that the drug decreases the growth rate. For $i\in \{1,\cdots, S\}$, the cell count of subpopulation $i$ at time $t$ and dose levels $d$ is modeled as
\begin{equation}
    \label{eq:subpopulation dynamic}
    f_i(t,d) = p_i X(0) \exp[t(\alpha_i + \log(H(d;b_i,E_i,n_i)))],
\end{equation}
where $p_i \in (0,1)$ satisfies $\sum_{i}p_i = 1$ and represents the initial proportion of subpopulation $i$, and $X(0)$ is the initial total cell count. While $X(0)$ is assumed to be a known quantity, PhenoPop estimates the following set of parameters from the data
\[\theta_{PP}(S) = \{p_i,\alpha_i,b_i,E_i,n_i; i\in \{1,\cdots,S\}\}.\]

To estimate these parameters, cell counts are collected at a set of time points $t\in \{t_1,\cdots,t_T\} = \mathcal{T}$ and drug dose levels $d\in \{d_1,\cdots,d_D\} = \mathcal{D}$; $R$ replications are performed for each possible time-dose combination. We denote the dataset as
\[\mathcal{X} = \{x_{t,d,r}; t\in \mathcal{T},d\in \mathcal{D},r \in \{1,\cdots,R\}\}.\]
PhenoPop then estimates the parameter set through the MLE process. 
In \cite{kohn2023phenotypic}, the authors assumed Gaussian noise with mean zero, featuring two possible variances based on the dose level and time of observation. 
For simplicity, we assume here the noise follows a mean-zero Gaussian distribution with a constant variance $\sigma^2$. 
The corresponding negative log-likelihood function for the observations $\mathcal{X}$ under the parameters set $\theta$ is given by
\begin{equation*}
    \label{eq: NLL PhenoPop}
    -\log(L(\theta,\sigma^2; \mathcal{X})) =  D R T \log\left(\frac{1}{\sqrt{2 \pi \sigma^2}}\right) + \frac{1}{2\sigma^2} \sum_{t\in \mathcal{T}} \sum_{d\in \mathcal{D}} \sum_{r = 1}^{R} (x_{t,d,r} - f(t,d;\theta))^2,
\end{equation*}
where $f(t,d;\theta) = \sum_{i = 1}^{S} f_i(t,d)$. To further simplify this formulation, we assume the observation noise $\sigma^2$ is known. So, the MLE problem is reduced to the following constrained least square problem
\begin{equation}
    \label{eq: PhenoPop Optimization}
    \hat{\theta} \in \argmin_{\theta\in \Theta} \sum_{t\in \mathcal{T}} \sum_{d\in \mathcal{D}} \sum_{r = 1}^{R} (x_{t,d,r} - f(t,d;\theta))^2,
\end{equation}
where $\hat{\theta}$ is an estimated parameter set and $\Theta$ is the feasible space of the parameters. In practice, the feasible space for the parameters is selected based on several biological assumptions. 

\subsubsection{Numerical Results}\label{sec: Deterministic HTS numerical results}
To evaluate the performance of each algorithm for solving \eqref{eq: PhenoPop Optimization}, we conducted a series of \textit{in silico} experiments. 
In each experiment, we generated 100 different datasets, each corresponding to a randomly selected true parameter set. 
For each dataset, we solved \eqref{eq: PhenoPop Optimization} starting from 20 different initial points.
Among those 20 results, we recorded the result with the best objective value as the solution obtained from each algorithm.
Details regarding the problem initialization and data simulation for these tests are located in Appendix \ref{appx: deterministic initialization}. 
It is important to note that in PhenoPop the primary source of randomness is the observation noise, which is less significant in this context; therefore, for the simulated data, we assume there is no observation noise, aiming to better illustrate each algorithms' performance. 
Thus, the optimal objective value should be close to zero; 
we consider values exceeding one indicative of poor estimation.
Our discussion of the experiments is divided into three sections based on the number of assumed subpopulations, $S$, in the model.
\\

\noindent\underline{Experiment with $S = 1$}\\

We begin with a preliminary experiment focusing on tumor dynamics involving a single dose-response curve. In this case, the parameter set consists of $\alpha,b,E,$ and $n$, without specification of the initial proportion parameter $p$. 
Figure \ref{fig: Deterministic S1} displays the three performance metrics discussed in Section~\ref{sec: Benchmark algorithms} for each of the five algorithms listed in Table~\ref{tab: Algorithm}. 
The box-plots display the results from 100 unique and independent datasets generated for the PhenoPop model. 

Evaluating each algorithm with our performance metrics, we see CRNAS demonstrated superior performance compared to the other algorithms in terms of the number of iterations required to achieve its best solution from one initial guess. 
Comparing total compute times, it is evident CRNAS required less time compared to IP and IP hess, while SQP and SQP grad required a similar amount of time to run for 20 different initial guesses; 
however, though SQP and SQP grad had comparable total compute times, they yielded less reliable solutions. 
CRNAS consistently delivered accurate estimates across all but one of the 100 instances of the model, as evidenced by its optimal obtained objective values consistently falling significantly below one in all but one test.
Both SQP and SQP grad on the other hand produced multiple poor estimations. 
Notably, without higher-order information, both SQP and IP could not reduce the objective value below $10^{-8}$. 
This demonstrates a handicap of utilizing the inexact higher-order information.
With exact higher-order information, both SQP grad and IP hess yielded optimal solutions lower than their inexact higher-order counterparts.
One might argue that since SQP grad has a comparable estimation quality to CRNAS, statistically speaking as evidenced by Figure~\ref{fig: Deterministic S1}, there is limited advantage to using CRNAS; however, this would be fallacious given the fact in about ten percent of the tests SQP grad produced poor estimates. 
IP hess produced accurate estimates, but required more time to obtain these solutions compared to CRNAS.
Overall, we see CRNAS performs on pair if not better than MATLAB's solvers based on our metrics. 
\\

\begin{figure}[H]
    \centering
    \includegraphics[width = \textwidth]{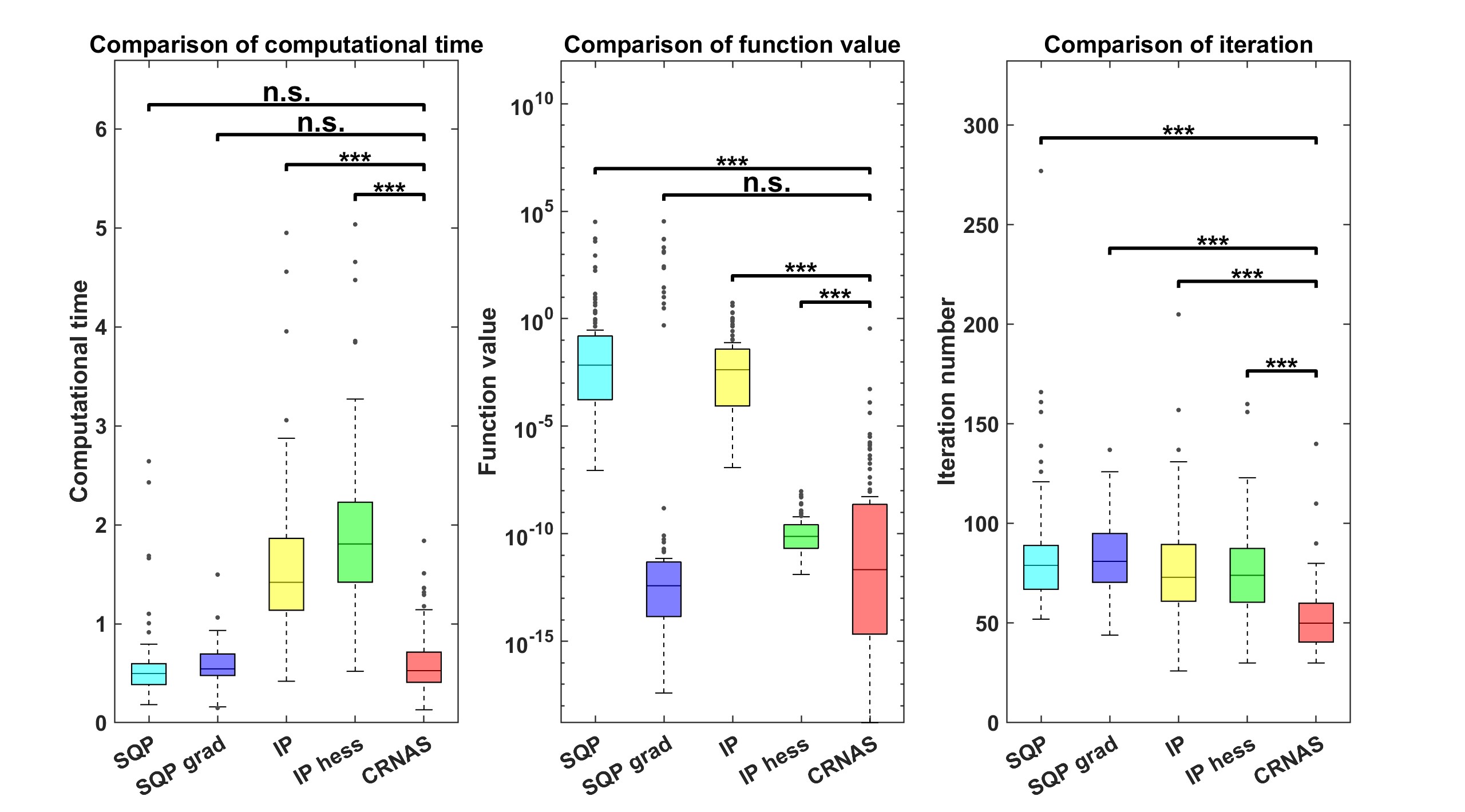}
    \caption{Comparison of the algorithms in Table \ref{tab: Algorithm} for solving \eqref{eq: PhenoPop Optimization} with a single subpopulation, i.e., $S = 1$. 
    The left panel records each algorithm's total compute time to solve each of the 100 instantiations of the model from 20 different initial points. 
    The middle panel records the best function value obtained by each algorithm from the 20 differential initial points for each of the 100 different experiments; the Y axis is in logarithmic scale. 
    The right panel records the number of iterations taken to obtain the optimal solution corresponding to the best objective values displayed in the middle panel. 
    Each boxplot indicates upper extremes, upper quantile, median, lower quantile, and lower extremes; the scattered dots represent outliers. 
    The significance bars indicate the p-values derived from the Wilcoxon rank sum test with significance levels such that $*** \leq 0.001 \leq ** \leq 0.01 \leq * \leq 0.05 \leq n.s.$.}
    \label{fig: Deterministic S1}
\end{figure}

\noindent\underline{Experiment with $S = 2$}\\

We now consider the case where two subpopulations exist in the tumor, each with a distinct drug response. 
The initial proportion parameter $p_i$ for each subpopulation is therefore reintroduced along with the corresponding equality constraint, $p_1 + p_2 = 1$; 
\begin{figure}[ht]
    \centering
    \includegraphics[width = \textwidth]{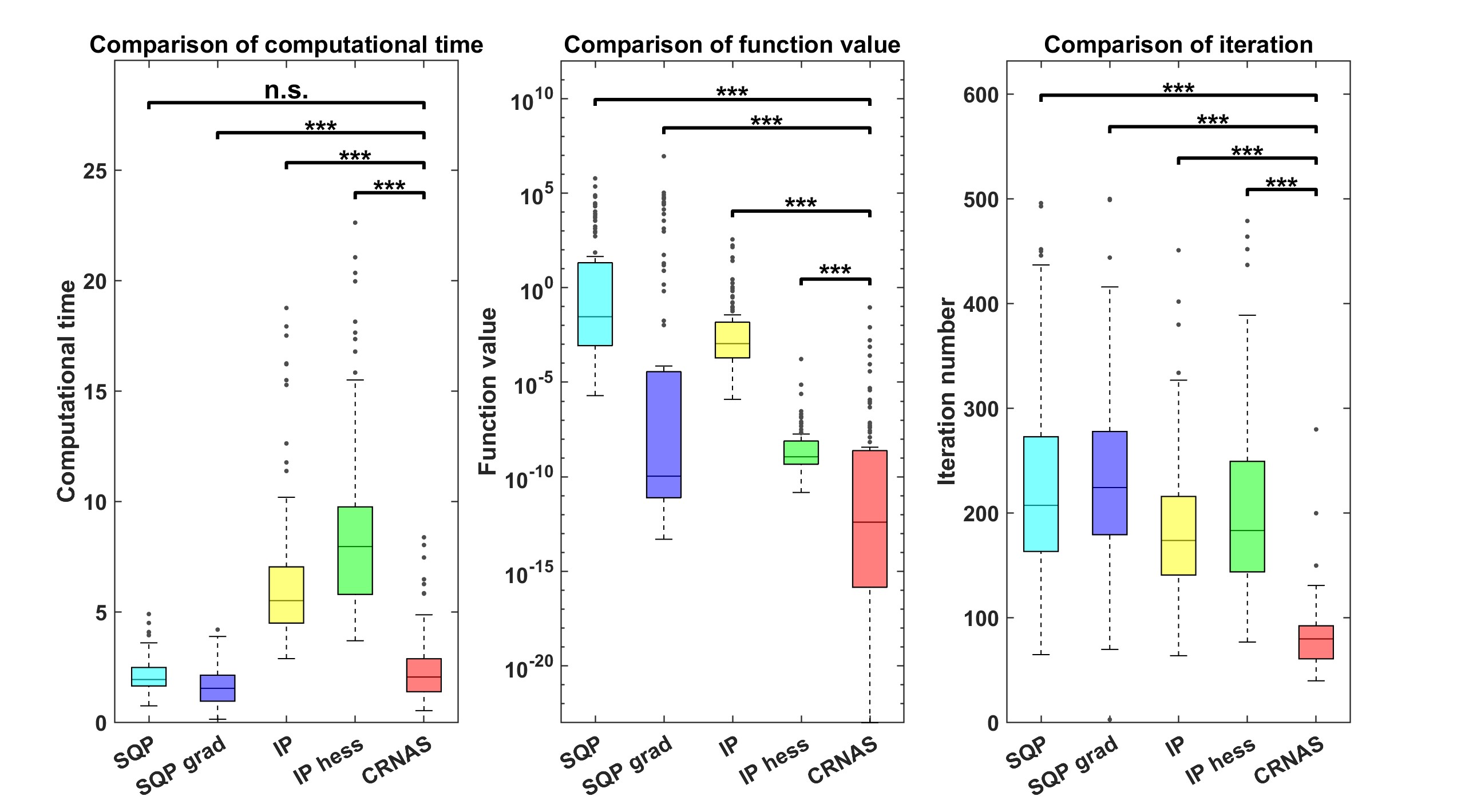}
    \caption{Comparison of the algorithms in Table \ref{tab: Algorithm} for solving \eqref{eq: PhenoPop Optimization} with $S = 2$ subpopulations; no upper bound constraint was present for $E_i$ and $n_i$ in these experiments. 
    The results are presented as before in Figure \ref{fig: Deterministic S1}.}
    \label{fig: Deterministic S2 uc}
\end{figure}
Figure \ref{fig: Deterministic S2 uc} provides the results for these experiments.

With two subpopulations present CRNAS maintained its advantage in terms of the number of iterations needed to obtain the optimal solution compared to the other algorithms. 
Adding an additional subpopulation increased the overall complexity of the model and this added complexity greatly limited SQP and SQP grad's ability to obtain accurate solutions. 
In more than 1/4 of the tests SQP's optimal computed objective values exceeded $1$, while more than half of the optimal solutions obtained by SQP grad were worse than the median of the optimal objective values obtained by CRNAS.
In contrast, CRNAS provided accurate estimations across all experiments for this more challenging problem; 
the performance of IP and IP hess was similar to the prior experiment. 

In seeking to understand the factors which limited the performance of SQP, we found imposing an upper bound constraint on some of the parameters 
improved the performance of SQP.
This technique requires researchers to heuristically determine an upper bound for all parameters. 
If the true parameter set is located within the narrowed zone, this significantly restricts the search region for the optimization algorithm and aids its location of better estimates.
The risk of employing this practice depends on the accuracy of the practitioner's prior knowledge since placing too small of an upper bound could make finding the true global optimal impossible as it will be outside the artificially constructed feasible region.
For the \textit{in silico} experiments, we have full knowledge about the true parameter set $\Theta^*(S)$, leading to an accurate optimization feasible region selection. 
Specifically, we set the optimal feasible region for $E_i$, and $n_i$ to be the interval $(0,100)$. The constraints on the other parameters remained the same. 

\begin{figure}
    \centering
    \includegraphics[width = \textwidth]{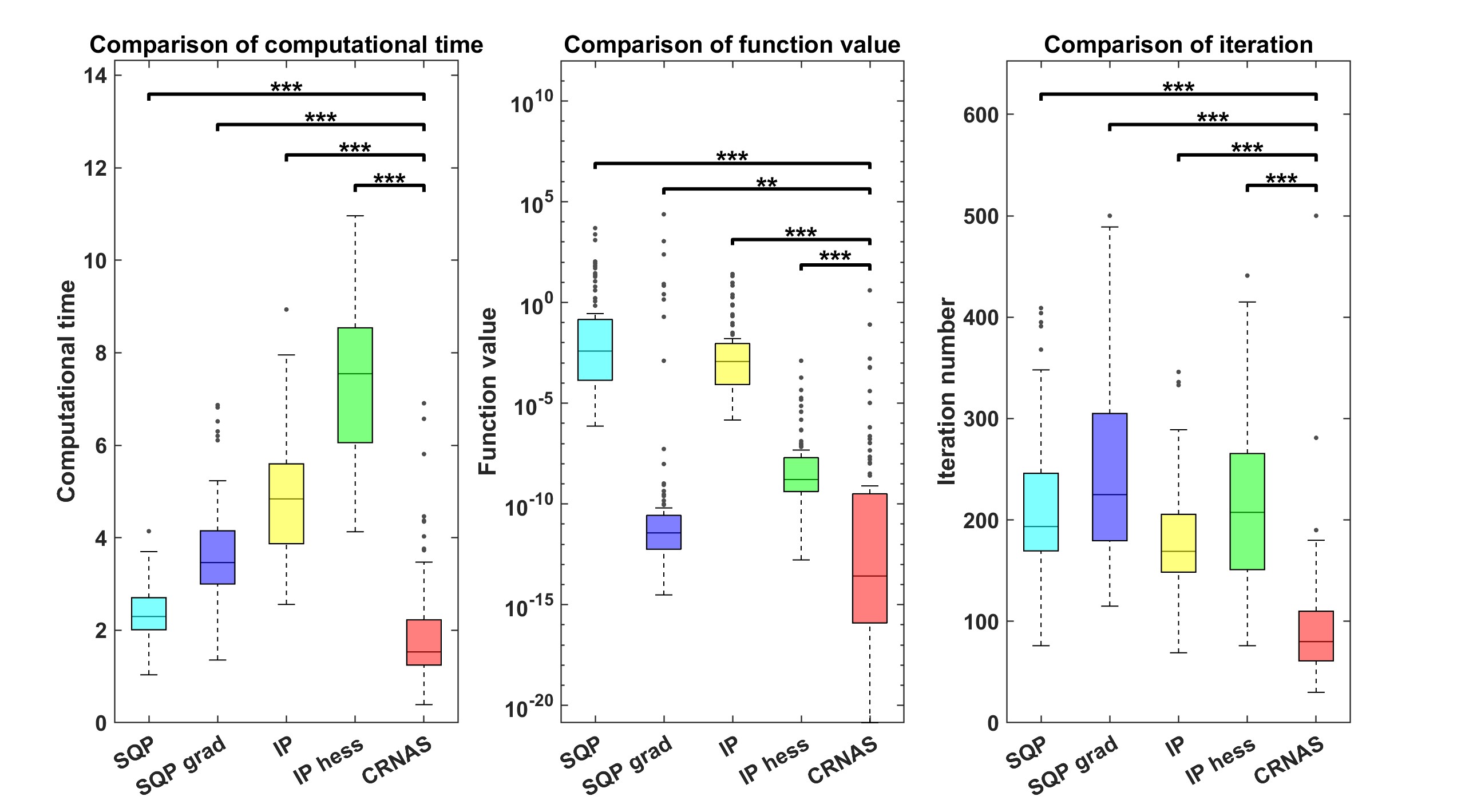}
    \caption{Comparison of the algorithms in Table \ref{tab: Algorithm} for solving \eqref{eq: PhenoPop Optimization} with $S = 2$ subpopulations. The `box constraint' is applied to each parameter. The results are presented as in Figure \ref{fig: Deterministic S1}.}
    \label{fig: Deterministic S2 wc}
\end{figure}

Using the new upper bound constraints on $E_i$ and $n_i$ another set of experiments was conducted; 
the results are provided in 
Figure \ref{fig: Deterministic S2 wc}. 
With these constraints, SQP grad performed in a similar fashion to the experiments with one subpopulation;
however, it required more total compute time compared to the outcomes depicted in Figure \ref{fig: Deterministic S2 uc}. 
Consequently, CRNAS significantly outperformed the other methods in both total compute time and number of iterations.
Additionally, CRNAS can provide robust estimates even without specifying upper bounds for all parameters, as shown in Figure \ref{fig: Deterministic S2 uc}.
So, a practitioner does not need to estimate potentially unknown upper bounds on parameters for their models in order for CRNAS to obtain strong parameter estimates.
\\

\noindent\underline{Experiment with $S > 2$}\\

\label{sec: experiment with S>2}

Lastly, we examined two more cases with $S = 3$ and $S = 5$. 
The majority of the conditions for generating the data remain unchanged from the experiments with $S = 2$; 
however, to accommodate setting $S = 3$ and $S = 5$, we employed a new scheme to select the EC50 values for each subpopulation and the dose levels $\mathcal{D}$ to ensure the identifiability of each parameter set; for details see Appendix \ref{appx: stochastic initialization}.

\begin{figure}[ht]
    \centering
    \includegraphics[width=\linewidth]{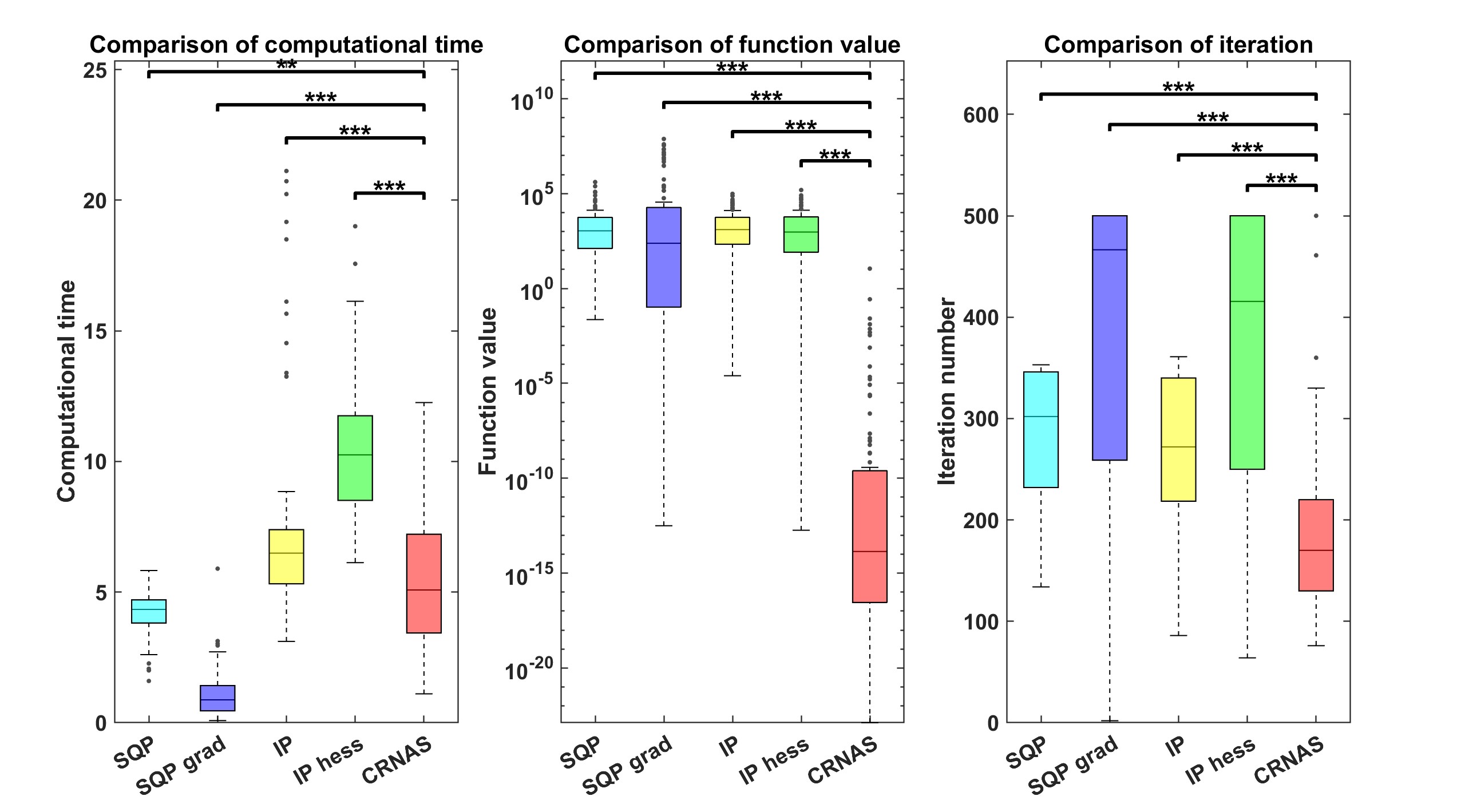}
    \caption{Comparison of the algorithms in Table \ref{tab: Algorithm} for solving \eqref{eq: PhenoPop Optimization} with $S = 3$ subpopulations. The results are presented as in Figure \ref{fig: Deterministic S1}.}
    \label{fig: PhenoPop UC experiment S3}
\end{figure}

\begin{figure}[ht]
    \centering
    \includegraphics[width=\linewidth]{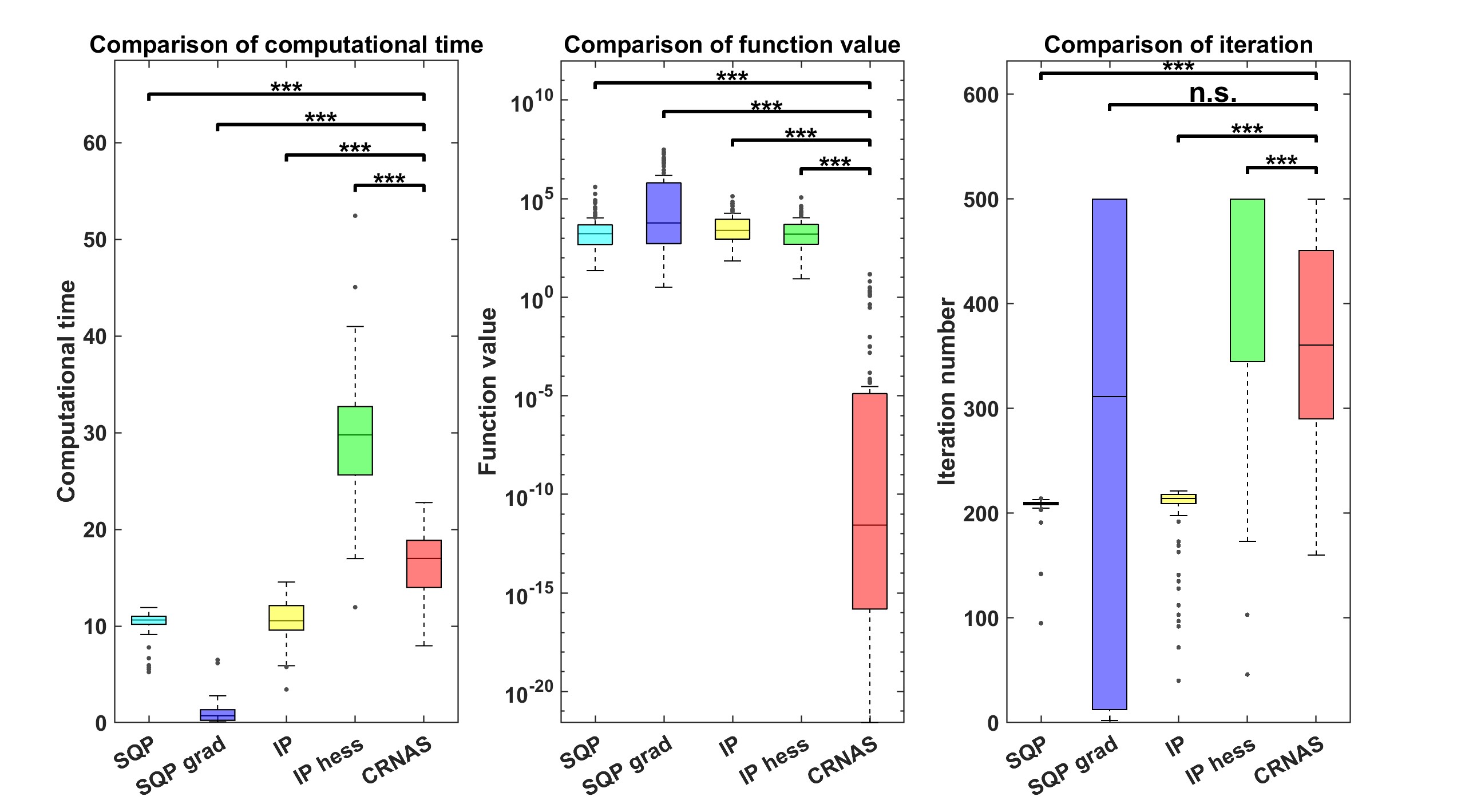}
    \caption{Comparison of the algorithms in Table \ref{tab: Algorithm} for solving \eqref{eq: PhenoPop Optimization} with $S = 5$ subpopulations. The results are presented as in Figure \ref{fig: Deterministic S1}.}
    \label{fig: PhenoPop UC experiment S5}
\end{figure}

The results for $S = 3$ and $S = 5$ are shown in Figures \ref{fig: PhenoPop UC experiment S3} and \ref{fig: PhenoPop UC experiment S5} respectively. 
When $S = 3$, it is evident that IP hess, which in the prior experiments produced very accurate solutions, 
located almost zero accurate estimates with nearly all its computed optimal objective values being larger than one. 
Some of the poor estimates produced by IP hess and SQP grad might be due to the algorithms reaching the maximum iteration limit of 500, see Appendix \ref{appx: stopping criteria} for details on termination criteria, but the algorithms produced poor estimates even when the iteration limit was not reached; as supported by the fact in more than half of the experiments with three subpopulations SQP grad and IP hess terminated before the maximum number of iterations was reached. 
CRNAS however consistently provided accurate estimates for both three and five subpopulations. 
We propose that this advantage stems from the inherent capability of CRNAS to automatically avoid saddle points and converge to second-order stationary points, avoiding suboptimal first-order stationary points which the other methods are potentially converging to. 

\subsection{Stochastic Drug-Affected Cell Proliferation Framework}\label{sec: Stochastic drug-affected framework}
\subsubsection{Problem Description}\label{sec: Stochastic drug-affected framework modeling}
To more accurately represent the changing variability in the data, \cite{wu2024using} extended the PhenoPop framework using linear birth-death processes to model the heterogeneous tumor cell populations. 
Instead of using a deterministic exponential growth model to represent the cell proliferation dynamics, they utilized the stochastic linear birth-death process. 
As a result, the relationship between the variance and the number of cells in the linear birth-death process can easily account for changing noise levels in the observations. 
Specifically, the authors assumed a cell in subpopulation $i$ (type-i cell) divides into two cells at rate $\beta_i \geq 0$ and dies at rate $\nu_i \geq 0$ stochastically. This means that during a short time interval $\Delta t >0$, a type-$i$ cell divides with probability $\beta_i \Delta t$ and dies with probability $\nu_i \Delta t$. The drug effect on type-$i$ cells is then modeled as a cytotoxic effect:
\[\nu_i(d) = \nu_i - \log(H(d;b_i,E_i,n_i)).\]
Note that this framework can also easily account for cytostatic effect. The authors denoted the stochastic process $X_i(t,d)$ as the number of cells in subpopulation $i$ at time $t$ under drug dose $d$ with mean and variance 
\begin{align*}
    \mathbb{E}[X_i(t,d)] &:= X(0)p_i \mu_i(t,d) = X(0) p_i \exp(t (\beta_i - \nu_i(d)))
\end{align*}
and
\begin{align*}
    \Var[X_i(t,d)] &:= X(0) p_i \sigma_i^2(t,d) = X(0)p_i \frac{\beta_i +\nu_i(d)}{\beta_i - \nu_i(d)} \left(\exp(2t(\beta_i-\nu_i(d))) - \exp(t(\beta_i-\nu_i(d))) \right).
\end{align*}
Since each subpopulation was assumed to grow independently, the total cell count $X(t,d) = \sum_{i = 1}^{S} X_i(t,d)$ is normally distributed with mean and variance
\begin{align*}
    \mathbb{E}[X(t,d)] &:= \mu(t,d) = \sum_{i = 1}^{S} X(0) p_i \mu_i(t,d)\\
    \Var[X(t,d)] &:=\sigma^2(t,d) = \sum_{i = 1}^{S} X(0) p_i \sigma_i^2(t,d),
\end{align*}
when the initial population $X(0)$ is large. 
Consequentially, the authors suggested the statistical framework
\begin{equation*}
    x_{t,d,r} = X^{(r)}(t,d) + Z_{t,d,r},
\end{equation*}
where $X^{(r)}(t,d)$ are independent and identical distributed (i.i.d.) copies of $X(t,d)$ for $r = 1,\cdots,R$, and $\{Z_{t,d,r}; d\in \mathcal{D},t\in \mathcal{T},r\in \{1,\cdots,R\}\}$ are i.i.d. normally distributed observation noise with mean $0$ and variance $c^2$. Thus, the model parameter set is now
\[\theta_{LBD}(S) = \{p_i,\beta_i,\nu_i,b_i,E_i,n_i,c;i\in \{1,\cdots,S\}\}.\]

This novel statistical framework aims to utilize the information in the variability of the data to predict the parameters $p_i,\beta_i,\nu_i,b_i,E_i$, and $n_i$ related to each subpopulation. As a result, there is a more complicated negative log-likelihood for the MLE process:
\begin{equation}
    \label{eq: NLL LBD}
    -\log(L(\theta_{LBD}(S);\mathcal{X})) = \sum_{t,d\in \mathcal{T},\mathcal{D}} R\log\left(\frac{1}{\sqrt{2\pi (\sigma^2(t,d)+c^2)}} \right) + \sum_{t,d\in \mathcal{T},\mathcal{D}} \sum_{r = 1}^{R} \frac{(x_{t,d,r} - \mu(t,d))^2}{2(\sigma^2(t,d) + c^2)}.
\end{equation}
The challenge in minimizing the negative log-likelihood function in \eqref{eq: NLL LBD} is its variance term. The variance now depends on the cell growth dynamics of each subpopulation which involves the Hill equation.   

\subsubsection{Numerical Results}
We next investigate CRNAS's performance on the MLE problem described by \eqref{eq: NLL LBD}. Details for the MLE problem initialization and data simulation are provided in Appendix \ref{appx: stochastic initialization}. 
In the stochastic model, verifying whether the solution reaches the global optimum is challenging, particularly when the parameter space is infinite. 
To address this, we use the likelihood of the true parameter set $\theta^*_{LBD}(S)$ as a targeted likelihood value. 
Specifically, we consider a solution to be good if its likelihood exceeds that of the true parameter set. 
In other words, a good solution should have a lower objective value (negative log-likelihood) than that of the true parameter set.
Since the likelihood of the true parameter $\theta^*_{LBD}(S)$ varies across different experiments, 
in order to create a unified metric across experiments we employed the relative likelihood for an estimator $\hat{\theta}_{LBD}(2)$ defined as
\begin{equation}
    \label{eq: relative likelihood}
    RL(\hat{\theta}_{LBD}(S); \theta^*_{LBD}(S)) := \frac{-\log(L(\hat{\theta}_{LBD}(S);\mathcal{X}(\theta^*_{LBD}(S)))}{-\log(L(\theta^*_{LBD}(S);\mathcal{X}(\theta^*_{LBD}(S)))},
\end{equation}
where the likelihood function $L(\cdot;\mathcal{X})$ is defined as in \eqref{eq: NLL LBD} and $\mathcal{X}(\theta^*_{LBD}(S))$ is the dataset generated by the true parameter set $\theta^*_{LBD}(S)$; typically this ratio is below $1$ for quality estimators.

As discussed in Section \ref{sec: Stochastic drug-affected framework modeling}, the advantage and novelty of the linear birth-death process framework is its variance structure can naturally explain the dynamic variance observed in real-world data. 
With mild assumptions on the cell growth dynamics, the linear birth-death process framework is able to utilize the information in the variance of the data to obtain more robust estimates of the model's parameters. 
The trade-off however is that employing this variance structure further complicates the likelihood function, as the variance term also contains a Hill function. 
For simplicity, we consider the setting with $S = 2$ subpopulations. 
In Section \ref{sec: experiment with S>2}, we concluded that inexact gradient and Hessian computations may mislead IP and SQP to terminate at poor solutions. 
We thus focus on comparing IP hess, SQP grad, and CRNAS in these experiments; Figure~\ref{fig: Linear birth-death UC20} displays the results of the experiments.

\begin{figure}
    \centering
    \includegraphics[width = \textwidth]{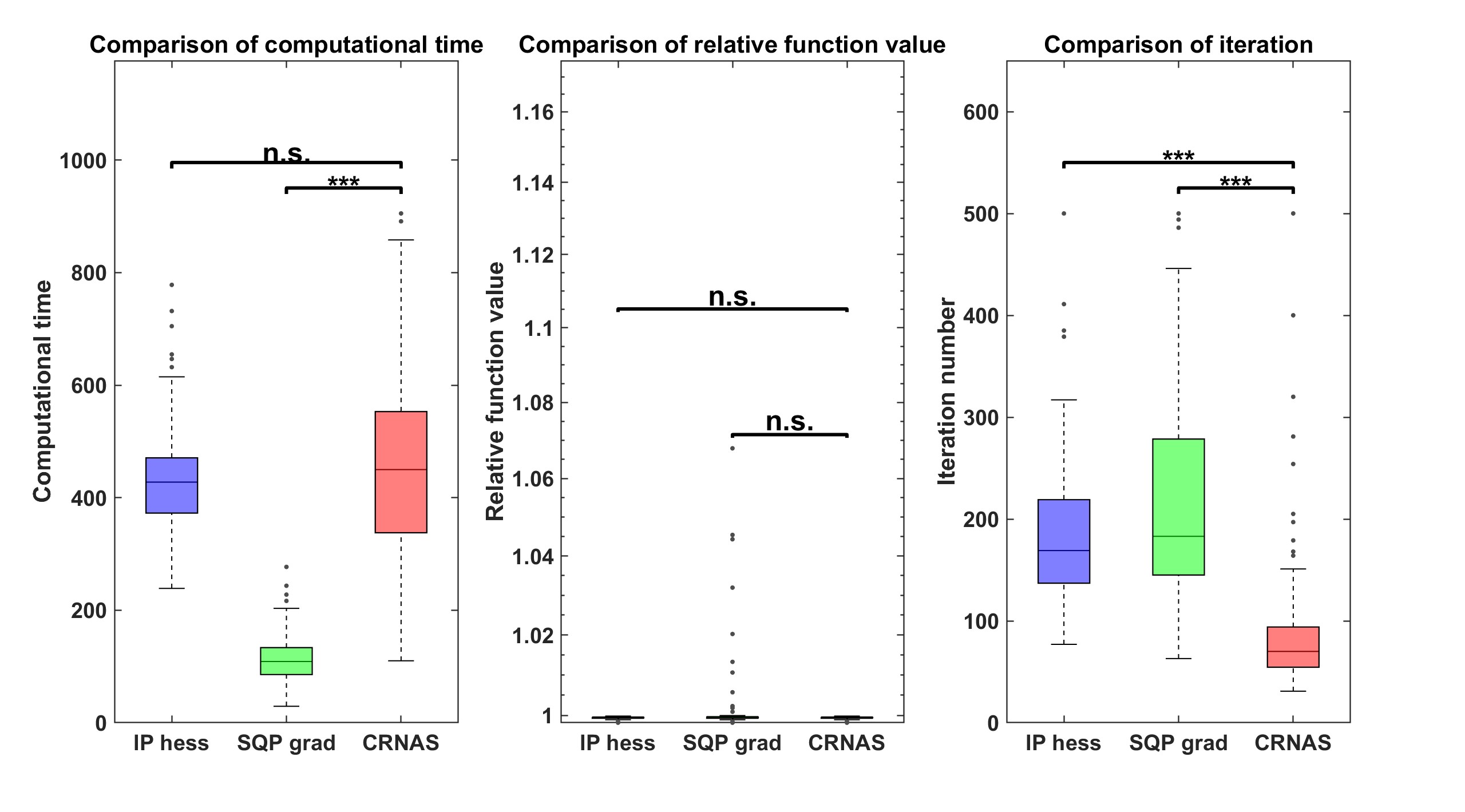}
    \caption{Comparison of three optimization algorithms in solving the MLE problem in \eqref{eq: NLL LBD} with $S = 2$ subpopulations. The results are presented as in Figure \ref{fig: Deterministic S1}.}
    \label{fig: Linear birth-death UC20}
\end{figure}

From the results, we note CRNAS still has the best convergence rate in terms of number of iterations; 
however, the total compute time of CRNAS was comparable to IP hess and longer than SQP grad. 
This is primarily because, for this problem, evaluating the Hessian matrix is significantly more expensive to compute than evaluating the gradient. 
Though SQP grad solved relatively quickly, many of the solutions were of low quality. 
Specifically, 14 of the 100 solutions obtained by SQP grad had a relative likelihood greater than 1, while all of the estimations obtained from IP hess and CRNAS had a relative likelihood less than 1. To illustrate this, we present an example estimation in Table \ref{tab: stochastic illustrative example}. 
In this example, it is evident SQP grad fails to estimate the parameters for the second subpopulation accurately, as the estimations for $E_2$ and $n_2$ are completely incorrect. 
This example reiterates the necessity of including the exact computation of higher-order information for better-quality solutions. 
Additionally, it supports our claim that when the relative likelihood is below $1$, the estimation quality is high.

\begin{table}[ht]
    \centering
    \small
    \begin{tabular}{|c|c|c|c|c|c|c|c|c|c|c|c|c|c|}
    \hline
         & $p_1$ & $\beta_1$ & $\nu_1$ & $b_1$ & $E_1$ & $n_1$ & $p_2$ & $\beta_2$ & $\nu_2$ & $b_2$ & $E_2$ & $n_2$  & $RL$ \\
        \hline
    $\theta^*(2)$ & 0.36 & 0.41 &  0.31 & 0.8 & 0.08 &  2.05 & 0.64 & 0.72 & 0.67 & 0.95 & 0.81 & 4.39  & 1\\
        \hline
    IP hess $\hat{\theta}(2)$  &  0.35 & 0.42 & 0.33 & 0.79 &  0.09 & 1.96 & 0.65 & 0.71 & 0.67 & 0.95 & 0.83 & 4.42  & 0.9997   \\
         \hline
    SQP grad $\hat{\theta}(2)$ & 0.38 & 0.43 & 0.34 & 0.8 &  0.09 & 1.82 & 0.62 & 0.73 & 0.68 & 0.95 & \textbf{1.19} & \textbf{36.51}  & 1.0009\\
        \hline
    CRNAS $\hat{\theta}(2)$ & 0.35 & 0.42 & 0.33 & 0.79 &  0.09 & 1.96 & 0.65 & 0.71 & 0.67 & 0.95 & 0.83 & 4.42  & 0.9997  \\
    \hline
    \end{tabular}
    \caption{An illustrative example of the results in Figure \ref{fig: Linear birth-death UC20}. The dataset was generated based on the true parameter set $\theta^*(2)$ and the estimations of IP hess, SQP grad, and CRNAS are shown. SQP grad incorrectly estimated both $E_2$ and $n_2$; these estimates are bolded. $RL$ in the last column is the relative likelihood defined in \eqref{eq: relative likelihood}.}
    \label{tab: stochastic illustrative example}
\end{table}


\section{Case Study II: Heterogeneous Logistic Estimation}\label{sec: case study: mixed logistic growth}
\subsection{Problem Description}
The logistic growth model is popular for studying population growth with a carrying capacity. 
Recently, there has been significant interest in an extended version of the logistic growth model, which aims to represent cell populations consisting of heterogeneous subpopulations  \cite{jin2018extended}. 
These subpopulations can exhibit varying growth dynamics, such as differing growth rates and carrying capacities.  
For simplicity here we study a linear combination of logistic growth functions with varying parameter sets for each subpopulation to model a mixture of heterogeneous populations growing according to a logistic function. 
Note, we do not assume the populations are interacting via their carrying capacity as that would necessitate the numerical solution of a nonlinear system of differential equations which is outside the focus of this manuscript.

In particular, we start with the parameters of $S$ subpopulations: $\theta_{LG}(S) = (p_i,\alpha_i,\beta_i),i = 1,\cdots,S$, and we then model the total cell count dynamics with respect to time as
\begin{equation}
\label{eq: heterogeneous logistic model}
F_{LG}(t;\theta_{LG}(S)) = \sum_{i = 1}^{S} f_i(t;p_i,\alpha_i,\beta_i) = \sum_{i = 1}^{S} F_{LG}(0)p_i\frac{1}{1+e^{-\alpha_i t + \beta_i}},
\end{equation}
where $F_{LG}(0)$ is the initial total cell count and $p_i$ is the initial proportion of each subpopulation, satisfying an equality constraint $\sum_{i = 1}^{S} p_i = 1$. 
Compared to the widely used three-parameter logistic growth model, we normalize the maximum carrying capacity for each subpopulation to 1, ensuring the identifiability of each parameter. 
We retain another two parameters, $(\alpha_i,\beta_i)$, to capture the growth behavior of type $i$ population at the inflection point.
The parameter $\alpha_i$ encodes the maximum growth rate that the type $i$ population can achieve around the inflection point, while the parameter $\beta_i$ relates to the shift in time of the inflection point. 
From a modeling perspective, there are no restrictions on $\alpha_i$ and $\beta_i$. 
However, we assume $\alpha_i \geq 0$ for all $i = 1,\cdots,S$ for a growing population and select a distinct $\beta_i$ for each subpopulation to ensure these parameters are identifiable.
The detailed selection method is specified in Appendix \ref{appx: heterogeneous logistic model details}.
For this experiment, we assume a zero observation noise. 
Thus, for a dataset observed at a given set of time points $\mathcal{T}$, 
$\mathcal{X} = \{x_{t}; t\in \mathcal{T}\}$,
we solve the least squares problem 
\begin{equation}
    \label{eq: logistic least square}
    \hat{\theta}(S) = \argmin_{\theta(S) \in \Theta(S)} \sum_{t\in \mathcal{T}} (x_{t} - F_{LG}(t; \theta(S)))^2
\end{equation}
to estimate the true parameter set $\theta^*_{LG}(S)$. 

\subsection{Numerical Results}
We compared the performance of all the algorithms in Table \ref{tab: Algorithm} for solving the least squares problem described in \eqref{eq: logistic least square}. 
Similar to the previous case study, 
we compared the methods based on the described metrics in Section~\ref{sec: Benchmark algorithms}.
Experimental details are provided in Appendix \ref{appx: heterogeneous logistic model details}. 


\begin{figure}
    \centering
    \includegraphics[width = \textwidth]{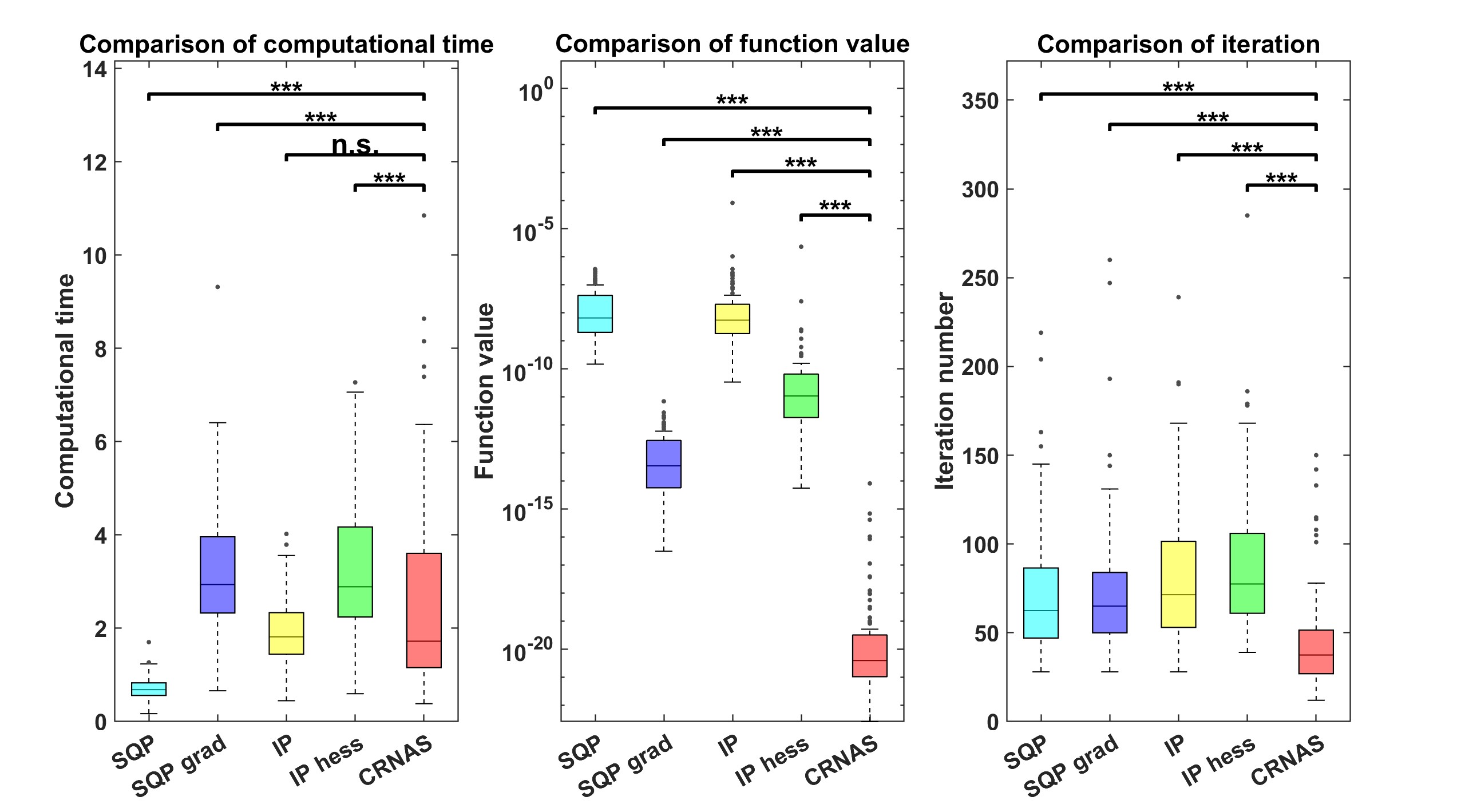}
    \caption{Comparison of the algorithms in Table \ref{tab: Algorithm} for solving the least square problem in \eqref{eq: logistic least square} with $S = 2$ subpopulations. The results are presented as in Figure \ref{fig: Deterministic S1}.}
    \label{fig: Logistic S2}
\end{figure}

From Figure \ref{fig: Logistic S2}, we see again CRNAS boasted the smallest number of iterations required to obtain the optimal solution. 
Furthermore, the middle panel in Figure \ref{fig: Logistic S2} 
shows CRNAS
yielded the lowest best objective values among all five methods. 
Based on the advantages observed in both the number iterations and objective value, we see CRNAS outperformed four different implementations of state-of-the-art algorithms in solving this highly nonlinear and non-convex parameter estimation problem. 

\section{Conclusion}

\label{sec: Conclusion}
This paper introduced a novel optimization algorithm called cubic regularized Newton based on affine scaling (CRNAS) for solving nonlinear and non-convex constrained optimization problems. 
CRNAS combines the best qualities of Nesterov's cubic regularized Newton's method and the classical affine scaling method for linear programming to produce a new approach with convergence guarantees to second-order stationary points.
The central motivator for the design of our algorithm was 
the parameter estimation problems which arise in mathematical biology.
These problems produce exceptionally challenging constrained optimization models and were able to 
demonstrate CRNAS competes with state-of-the-art algorithms on these problems.
%
%

We conducted a comparative analysis of CRNAS with two state-of-the-art local optimization methods implemented in MATLAB's commercial optimization solver \textit{fmincon}. 
This analysis was performed across two case studies, and our findings consistently demonstrated CRNAS outperformed \textit{fmincon}  
by delivering generally more accurate solutions in fewer iterations and with comparable, if not better, total compute times;
however, it is important to note that the computational effort required to compute second-order information may diminish the advantage of CRNAS. 
Since the Hessian matrix must be computed at each iteration of CRNAS,
we recommend using CRNAS when accurate second-order information can be obtained without expending an excessive amount of computational resources compared to computing first-order information.


\section*{Acknowledgements}
The work of K.\ Leder, S.\ Zhang,  C.\ Wu and N.\ Wang was partially supported by NSF Award CMMI 2228034; C.\ Garner is supported by the NSF-GRFP under Grant No.\ 2237827.

\appendix 
\renewcommand{\theequation}{A-\arabic{equation}} 

  \setcounter{equation}{0}  
  
  \section{Appendix}  
The supplemental materials included here are structured as follows: Appendix~\ref{complex} provides the proof of Theorem~\ref{thm:iteration_complexity} which provides the convergence rate of CRNAS; Appendix~\ref{subproblem} describes how to solve the crucial subproblem in CRNAS; Appendix~\ref{first-complexity} describes a first-order version of CRNAS called first-order affine scaling and provides a theoretical analysis of its performance; Appendices~\ref{appx: deterministic initialization}, \ref{appx: stochastic initialization}, and \ref{appx: heterogeneous logistic model details} provide additional details about of our numerical experiments; lastly, Appendix~\ref{appx: stopping criteria} states the termination criteria utilized by all of the algorithms in our experiments.  

\subsection{Complexity Analysis of CRNAS}\label{complex}
To begin, we quote Lemma 3 in \cite{he2023newton} which shall prove crucial in our analysis.
\begin{lemma}\label{taylor}
Based on Assumption \ref{lip}, the following inequalities hold for all $x, y \in \mbox{\rm int}(\mathcal{K})$:
\[\left\|\nabla L(y)-\nabla L(x)-\nabla^2 L(x)(y-x)\right\|_x^* \leq \frac{1}{2} \beta\|y-x\|_x^2,\]
    \[L(y) \leq L(x)+\nabla L(x)^\top(y-x)+\frac{1}{2}(y-x)^\top \nabla^2 L(x)(y-x)+\frac{1}{6} \beta \|y-x\|_x^3.\]
\end{lemma}

Observe that the affine scaling algorithm is to minimize the majorization. We show the function value decreases can be lower bounded by the norm of $\nabla L(\theta^{k+1})$ projected in the linear space in the following two lemmas.

\begin{lemma}\label{decrease}
    If $M\geq 2\beta$, then the following inequality holds:
    \[L(\theta^k)-L(\theta^{k+1})\geq \frac{M}{12}\|\theta^{k+1}-\theta^k\|_{\theta^k}^3.\]
\end{lemma}

\begin{proof}
Since $\theta^{k+1}$ is the optimal solution of the subproblem, we have
    \begin{eqnarray}
    L(\theta^k)&\geq &L(\theta^k)+\left\langle \nabla L(\theta^k), \theta^{k+1}-\theta^k\right\rangle+\frac{1}{2}\left\langle \nabla^2 L(\theta^k)(\theta^{k+1}-\theta^k), \theta^{k+1}-\theta^k\right\rangle+\frac{M}{6}\|\theta^{k+1}-\theta^k\|_{\theta^k}^3\nonumber\\
    &\geq &L(\theta^{k+1})+\frac{M}{12}\|\theta^{k+1}-\theta^k\|_{\theta^k}^3,\nonumber
    \end{eqnarray}
where the last step is according to Lemma \ref{taylor}.
\end{proof}

Because of the linear constraint $A\theta=b$, we expect $(\nabla L(\theta^{k+1}))^\top d$ to converge to zero after affine scaling, where $d$ is any given feasible direction.

\begin{lemma} \label{lemma1}
    Assume the constraint $\|\theta-\theta^k\|_{\theta^k}\leq 1-\alpha$ in the subproblem is inactive, namely $\|\theta^{k+1}-\theta^k\|_{\theta^k}< 1-\alpha$. Then, for any $d$, such as $Ad=0$, we have
    \[|(\nabla L(\theta^{k+1}))^\top d| \leq  \frac{M+\beta}{2}\|\theta^{k+1}-\theta^k\|_{\theta^k}^2\|d\|_{\theta^k}\]
\end{lemma}
\begin{proof}
Because the constraint $\|\theta-\theta^k\|_{\theta^k}\leq 1-\alpha$ is inactive, any $d$ satisfying $Ad=0$ must be a feasible direction. According to the optimality condition, we have
\[\left(\nabla L(\theta^k)+\nabla^2 L(\theta^k)(\theta^{k+1}-\theta^k)\right)^\top d+\frac{M}{2}\|\theta^{k+1}-\theta^k\|_{\theta^k}(\theta^{k+1}-\theta^k)^\top(\nabla^2 B(\theta^k))d=0,\]
which implies that
\[|\left(\nabla L(\theta^k)+\nabla^2 L(\theta^k)(\theta^{k+1}-\theta^k)\right)^\top d| = \frac{M}{2}\|\theta^{k+1}-\theta^k\|_{\theta^k}|(\theta^{k+1}-\theta^k)^\top (\nabla^2 B(\theta^k))d|\leq \frac{M}{2}\|\theta^{k+1}-\theta^k\|_{\theta^k}^2\|d\|_{\theta^k}\]
According to Lemma \ref{taylor}, we have
    $\left\|\nabla L(\theta^{k+1})-\nabla L(\theta^k)-\nabla^2 L(\theta^k)(\theta^{k+1}-\theta^k)\right\|_{\theta^k}^* \leq  \frac{\beta}{2}\|\theta^{k+1}-\theta^k\|_{\theta^k}^2$. So,
    \[|\left(\nabla L(\theta^{k+1})-\nabla L(\theta^k)-\nabla^2 L(\theta^k)(\theta^{k+1}-\theta^k)\right)^\top d|\leq \frac{\beta}{2}\|\theta^{k+1}-\theta^k\|_{\theta^k}^2\|d\|_{\theta^k}\]
    Combining the above two, we get the inequality.
\end{proof}

We give the lemma below to show the second-order optimality of the solution obtained from the algorithm.
\begin{lemma} \label{lemma2}
Assume the constraint $\|\theta-\theta^k\|_{\theta^k}\leq 1-\alpha$ in the subproblem is inactive, namely $\|\theta^{k+1}-\theta^k\|_{\theta^k}< 1-\alpha$. Then, for any $d$, such as $Ad=0$, we have
    \[d^\top \nabla^2 L(\theta^{k+1})d\geq -(M+\beta)\|\theta^{k+1}-\theta^k\|_{\theta^k}d^\top \nabla^2 B(\theta^k)d\]
\end{lemma}

\begin{proof}
Because the constraint $\|\theta-\theta^k\|_{\theta^k}\leq 1-\alpha$ is inactive, any $d$ satisfying $Ad=0$ must be a feasible direction. According to the second-order optimality condition, we have
\[d^\top (\nabla^2 L(\theta^k)+\frac{M}{2}(\nabla^2 B(\theta^k))(\theta^{k+1}-\theta^k)(\theta^{k+1}-\theta^k)^\top (\nabla^2 B(\theta^k))/\|\theta^{k+1}-\theta^k\|_{\theta^k}+\frac{M}{4}\|\theta^{k+1}-\theta^k\|_{\theta^k}\nabla^2 B(\theta^k))d\geq 0\]

So, we observe that
\begin{eqnarray}
    &&d^\top \left(\nabla^2 L(\theta^k)+M\|\theta^{k+1}-\theta^k\|_{\theta^k}\nabla^2 B(\theta^k)\right)d\\
    &\geq& d^\top \left(\nabla^2 L(\theta^k)+\frac{M}{2}(\nabla^2 B(\theta^k))(\theta^{k+1}-\theta^k)(\theta^{k+1}-\theta^k)^\top(\nabla^2 B(\theta^k))/\|\theta^{k+1}-\theta^k\|_{\theta^k}+\frac{M}{4}\|\theta^{k+1}-\theta^k\|_{\theta^k}\nabla^2 B(\theta^k)\right)d\nonumber\\
    &\geq& 0.\nonumber
\end{eqnarray}
According to Assumption \ref{lip}, we have
    \[\nabla^2 L(\theta^{k+1})\succeq \nabla^2 L(\theta^k)-\beta\|\theta^{k+1}-\theta^k\|_{\theta^k}\nabla^2 B(\theta^k).\]
Combining the above two inequalities, we have
\[d^\top \nabla^2 L(\theta^{k+1})d\geq -(M+\beta)\|\theta^{k+1}-\theta^k\|_{\theta^k}d^\top\nabla^2 B(\theta^k)d.\]
\end{proof}


According to Lemma \ref{decrease}, we have
\[L(\theta^k)-L(\theta^{k+1})\geq \frac{M}{12}\|\theta^{k+1}-\theta^k\|_{\theta^k}^3\geq \frac{M\eta^3}{12}, k=0,1,2,...,K-2,\]
which implies the decrease guarantee in each step so that the algorithm will stop in finite steps. Telescoping the above inequalities, we have
\[\frac{M\eta^3}{12}(K-1)\leq \sum_{k=0}^{K-2}(L(\theta^k)-L(\theta^{k+1}))=L(\theta^0)-L(\theta^{K-1})\leq L(\theta^0)-L^*.\]

Therefore, we can give an upper bound for the number of iterations $K\leq 12(L(\theta^0)-L^*)\eta^{-3}+1$. According to Lemma \ref{lemma1} and \ref{lemma2} with $M\geq 2\beta$ and $\eta\leq 1-\alpha$, we have the following bounds.

For any $d$, such as $Ad=0$, we have
\[|(\nabla L(\theta^K))^\top  d|\leq \frac{M+\beta}{2}\|\theta^{K}-\theta^{K-1}\|_{\theta^{K-1}}^2\|d\|_{\theta^{K-1}}\leq
M\eta^2\|d\|_{\theta^{K-1}} \]
and 
\[d^\top \nabla^2 L(\theta^{K}) d\geq  -(M+\beta)\|\theta^{K}-\theta^{K-1}\|_{\theta^{K-1}}d^\top \nabla^2 B(\theta^{K-1})d\geq -2M\eta d^\top \nabla^2 B(\theta^{K-1})d\]




Given that the right hand sides of the above two inequalities also involve $\theta^{K-1}$, we site the following lemma to change them into $\theta^{K}$.
\begin{lemma} (Theorem 2.1.1 in \cite{nesterov1994interior})
    Let $B$ be a-self-concordant on $\mathcal{K}$ and let $\theta_0\in  \mbox{\rm int}(\mathcal{K})$, then for any $\theta\in  \mbox{\rm int}(\mathcal{K})$ and $\|\theta-\theta_0\|_{\theta_0}<1$, it holds that $(1-\|\theta-\theta_0\|_{\theta_0})^2 \leq u^\top \nabla^2 B(\theta_0) uu^\top \nabla^2 B(\theta) u \leq \frac{u^\top \nabla B(\theta_0) u}{(1-\|\theta-\theta_0\|_{\theta_0})^2}$.
\end{lemma}

Since $\|\theta^{K}-\theta^{K-1}\|_{\theta^{K-1}}<\eta\leq 1-\alpha$, we have $\alpha^2\nabla^2 B(\theta^{K})\preceq\nabla^2 B(\theta^{K-1})\preceq \frac{1}{\alpha^2}\nabla^2 B(\theta^{K})$. The above two inequalities imply 
\[|(\nabla L(\theta^K))^\top  d|\leq
M\eta^2\|d\|_{\theta^{K-1}}\leq M\eta^2\alpha^{-1}\|d\|_{\theta^{K}} \]
and 
\[
d^\top \nabla^2 L(\theta^{K}) d\geq -2M\eta d^\top \nabla^2 B(\theta^{K-1})d\geq -2M\eta\alpha^{-2} d^\top \nabla^2 B(\theta^{K})d.
\]
We conclude the following theorem.
\begin{theorem} \label{iteration_complexity}
    Apply Algorithm CRNAS \ref{alg} with Assumption \ref{lip} and let $M=2\beta$ and $\eta=\min\{1-\alpha,\epsilon^{-1/2}\alpha^{1/2}M^{-1/2},\frac{1}{\sqrt{2}}\epsilon^{-1/2}\alpha^{2}M^{-1}\}$. Then the algorithm will stop at $K\leq 12(L(\theta^0)-L^*)\eta^{-3}+1$, and $\theta^K$ is a second-order $\epsilon$-optimal solution.
\end{theorem}

\subsection{Solving the Subproblem}\label{subproblem}
In each step, we solve the following subproblem.
   \begin{eqnarray}
        \theta^{k+1}&=&\arg \min _{\theta:A\theta=b,\|\theta-\theta^k\|_{\theta^k}\leq 1-\alpha} \left(\left\langle\nabla L\left(\theta^k\right), \theta-\theta^k\right\rangle +\frac{1}{2}\nabla^2L\left(\theta^k\right)[\theta-\theta^k]^2+\frac{M}{6}\left\|\theta-\theta^k\right\|_{\theta^k}^3\right).\nonumber
    \end{eqnarray}
     Let $T$ be an orthogonal matrix whose columns form a basis of the linear space $A\theta=0$ {\color{blue} nontrival }. Then, the linear constraint $A\theta=b$ can be replaced by $\theta=\theta^k+T\theta'$.
Let $\left\|\theta-\theta^k\right\|_{\theta^k}=\|(\nabla^2 B(\theta^k))^{1/2}(\theta-\theta^k)\|=\|(\nabla^2 B(\theta^k))^{1/2}T\theta'\|=\|(T^\top\nabla^2 B(\theta^k)T)^{1/2}\theta'\|$. The matrix $T^\top\nabla^2 B(\theta^k)T$ is invertible because $T^\top\nabla^2 B(\theta^k)T=((\nabla^2 B(\theta^k))^{1/2}T)^\top((\nabla^2 B(\theta^k))^{1/2}T)$ where $T$ is of full column rank and $\nabla^2 B(\theta^k)$ is of full rank. Then by letting $\bar{\theta}=(T^\top\nabla^2 B(\theta^k)T)^{1/2}\theta'$, the subproblem can be written as 
\[\min_{\|\bar{\theta}\|\leq 1-\alpha} m(\bar{\theta}):=g^\top \bar{\theta}+\frac{1}{2}\bar{\theta}^\top P\bar{\theta}+\frac{M}{6}\|\bar{\theta}\|^3\]
where $g$ is a vector and $P$ is a matrix with corresponding sizes.

Note in the subproblem of standard cubic Newton, we solve $\min_{\bar{\theta}}m(\bar{\theta})$, where we confine $\|\bar{\theta}\|\leq 1-\alpha$ in the above subproblem. If the solution to $\min_{\bar{\theta}}m(\bar{\theta})$  satisfies $\|\bar{\theta}^*\|\leq 1-\alpha$, then we get the solution. Otherwise, we introduce the following two lemmas to solve the subproblem.

\begin{lemma}\label{sub}
Let $\hat{\theta}\in\argmin_{\|\bar{\theta}\|\leq 1-\alpha} m(\bar{\theta})$ and $\bar{\theta}^*\in\min_{\bar{\theta}} m(\bar{\theta})$. If $\|\bar{\theta}^*\|>1-\alpha$, then we must have $\|\hat{\theta}\|=1-\alpha$.
\end{lemma}
\begin{proof}
    We cited part of the proof of Theorem 3.1 in \cite{cartis2011adaptive}. Suppose $\|\hat{\theta}\|<1-\alpha$, we will show $\hat{\theta}$ is a global minimum and no global solution will satisfy $\|\theta\|<1-\alpha$. Since $\hat{\theta}$ is an interior point of the set $\{\theta:\|\theta\|\leq 1-\alpha\}$, according to the first and second-order necessary optimality
conditions
\[g+(P+\hat{\lambda} I)\hat{\theta}=\textbf{0}\label{first}\]
and
\begin{equation}
w^\top \left(P+\hat{\lambda} I+\hat{\lambda} \left(\frac{\hat{\theta}}{\|\hat{\theta}\|}\right) \left(\frac{\hat{\theta}}{\|\hat{\theta}\|}\right)^\top\right) w\geq 0\label{sec}\end{equation}
for all vectors $w$, where $\hat{\lambda}=\frac{M}{2}\|\hat{\theta}\|$.
If $\|\hat{\theta}\|=0$, the second-order optimizality condition is $P+\hat{\lambda} I\succeq 0$. We will show for $\|\hat{\theta}\|\neq 0$, $P+\hat{\lambda} I\succeq 0$ still holds.

For $w^\top\hat{\theta}=0$, \eqref{sec} shows that $w^\top(P+\hat{\lambda} I)w\geq 0$. Consider the vector $w^\top\hat{\theta}\neq0$, the line
$\hat{\theta}+\alpha w$ intersects the ball of radius $\hat{\theta}$ at two points, $\hat{\theta}\neq \theta'$, where $\|\hat{\theta}\|=\|\theta'\|$. Without loss of generality, let $w=\theta'-\hat{\theta}$. Since $\hat{\theta}\in\argmin_{\|\bar{\theta}\|\leq 1-\alpha} m(\bar{\theta})$, we have
\begin{eqnarray}
    0&\leq& m(\theta')-m(\hat{\theta})\nonumber\\
    &=& g^\top(\theta'-\hat{\theta})+\frac{1}{2}\theta'^\top P \theta'-\frac{1}{2}\hat{\theta}^\top P \hat{\theta}\nonumber\\
    &=&\hat{\theta}^\top(P+\hat{\lambda} I)(\hat{\theta}-\theta')+\frac{1}{2}\theta'^\top P \theta'-\frac{1}{2}\hat{\theta}^\top P \hat{\theta}\nonumber\\
    &=&\hat{\lambda}\|\hat{\theta}\|^2-\hat{\lambda}\hat{\theta}^\top \theta'-
    \hat{\theta}^\top P\theta'+\frac{1}{2}\theta'^\top P \theta'+\frac{1}{2}\hat{\theta}^\top P \hat{\theta}\nonumber\\
    &=&\frac{1}{2}\hat{\lambda}\|\hat{\theta}\|^2+\frac{1}{2}\hat{\lambda}\|\theta'\|^2-\hat{\lambda}\hat{\theta}^\top \theta'-
    \hat{\theta}^\top P\theta'+\frac{1}{2}\theta'^\top P \theta'+\frac{1}{2}\hat{\theta}^\top P \hat{\theta}\nonumber\\
    &=&\frac{1}{2}\hat{\lambda} (\theta'-\hat{\theta})^\top (\theta'-\hat{\theta})+\frac{1}{2}(\theta'-\hat{\theta})^\top P (\theta'-\hat{\theta})\nonumber\\
    &=&\frac{1}{2}w^\top (P+\hat{\lambda} I)w\nonumber
\end{eqnarray}
Therefore, $P+\hat{\lambda} I\succeq 0$ always holds. Let $P$ has an eigendecomposition $P=U^\top \Lambda U$, where $\Lambda=diag(\lambda_1,\lambda_2,...,\lambda_n)$. Then, according to \ref{first}, we have
\[Ug+(\Lambda+\hat{\lambda} I)U\hat{\theta}=\textbf{0},\]
where $\hat{\lambda}=\frac{M}{2}\|\hat{\theta}\|$ and $\lambda_i+\hat{\lambda}\geq 0$, $i=1,2,...,n$. 

On the other hand, according to the first-order optimality condition for the global solution $\bar{\theta}^*$, we have 
\[g+(P+\bar{\lambda}^* I)\bar{\theta}^*=\textbf{0},\]
which is equivalent to 
\[Ug+(\Lambda+\bar{\lambda}^* I)U\bar{\theta}^*=\textbf{0},\]
where $\bar{\lambda}^*=\frac{M}{2}\|\bar{\theta}^*\|\geq \frac{M}{2}(1-\alpha)>\hat{\lambda}$. Therefore, we have 
$\lambda_i+\bar{\lambda}^*>\lambda_i+\hat{\lambda}\geq 0$, $i=1,2,...,n$. Thus, we can get the following equation
\[U\bar{\theta}^*=(\Lambda+\bar{\lambda}^* I)^{-1}(-Ug)=(\Lambda+\bar{\lambda}^* I)^{-1}(\Lambda+\hat{\lambda} I)U\hat{\theta},\]
which implies $|(U\bar{\theta}^*)_i|=|\frac{\lambda_i+\hat{\lambda}}{\lambda_i+\bar{\lambda}^*}(U\hat{\theta})_i|\leq |(U\hat{\theta})_i|$. Therefore, we have $\|U\bar{\theta}^*\|\leq \|U\hat{\theta}\|$. However, since $U$ is orthogonal,
\[\|U\bar{\theta}^*\|=\|\bar{\theta}^*\|\geq 1-\alpha>\|\hat{\theta}\|=\|U\hat{\theta}\|,\]
which causes a contradiction.
\end{proof}

\begin{lemma}
Let $\tilde{\theta}\in\arg\min_{\|\bar{\theta}\|\leq 1-\alpha} g^\top \bar{\theta}+\frac{1}{2}\bar{\theta}^\top P\bar{\theta}$ and $\bar{\theta}^*\in\min_{\bar{\theta}} m(\bar{\theta})$. If $\|\bar{\theta}^*\|> 1-\alpha$, then 
\[\tilde{\theta} \in\argmin_{\|\bar{\theta}\|\leq 1-\alpha} m(\bar{\theta})\]

\end{lemma}

\begin{proof}
    Suppose $\tilde{\theta}\notin \argmin_{\|\bar{\theta}\|\leq 1-\alpha} m(\bar{\theta})$. Let $\hat{\theta}\in \argmin_{\|\bar{\theta}\|\leq 1-\alpha} m(\bar{\theta})$. Then we have $m(\tilde{\theta})>m(\hat{\theta})$, which is
    \[g^\top \tilde{\theta}+\frac{1}{2}\tilde{\theta}^\top P\tilde{\theta}+\frac{M}{6}\|\tilde{\theta}\|^3>g^\top \hat{\theta}+\frac{1}{2}\hat{\theta}^\top P\hat{\theta}+\frac{M}{6}\|\hat{\theta}\|^3\]
    According to Lemma \ref{sub}, we have $\|\hat{\theta}\|=1-\alpha\geq \|\tilde{\theta}\|$. Therefore,
    \[g^\top \tilde{\theta}+\frac{1}{2}\tilde{\theta}^\top P\tilde{\theta}>g^\top \hat{\theta}+\frac{1}{2}\hat{\theta}^\top P\hat{\theta}+\frac{M}{6}(\|\hat{\theta}\|^3-\|\tilde{\theta}\|^3)\geq g^\top \hat{\theta}+\frac{1}{2}\hat{\theta}^\top P\hat{\theta},\]
    which contradicts with the fact $\tilde{\theta}\in\arg\min_{\|\bar{\theta}\|\leq 1-\alpha} g^\top \bar{\theta}+\frac{1}{2}\bar{\theta}^\top P\bar{\theta}$.
\end{proof}

Therefore, when solving the subproblem, we first solve the problem $\min_{\bar{\theta}}m(\bar{\theta})$ by the method introduced in Section 6.1 of \cite{cartis2011adaptive}. If $\|\bar{\theta}^*\|\leq 1-\alpha$, we find the solution $\bar{\theta}^*$. If not, we know the optimal solution is on the boundary, namely $\|\hat{\theta}\|= 1-\alpha$. Thus, we can skip the cubic term by adding the ball constraints. Only the following trust-region subproblem needs to be solved
\[\hat{\theta}=\arg\min_{\|\bar{\theta}\|\leq 1-\alpha} g^\top \bar{\theta}+\frac{1}{2}\bar{\theta}^\top P\bar{\theta},\]
which is discussed in  7.3 and 1.3 of \cite{conn2000trust}.

\subsection{A First-Order Version of CRNAS}\label{first-complexity}

Here we discuss a first-order version of CRNAS to solve \eqref{main:problem}. 
The construction of our first-order method is essentially the same as that for CRNAS with the exception a quadratic approximation of the function is utilized rather than a cubic approximation. 
Hence, the subproblem solved at each iteration of our first-order affine scaling method (FOAS) is
    \begin{eqnarray}
        \theta^{k+1}&=&\arg \min _{\theta:A\theta=b,\|\theta-\theta^k\|_{\theta^k}\leq 1-\alpha} \left(\left\langle\nabla L\left(\theta^k\right), \theta-\theta^k\right\rangle +\frac{M}{2}\left\|\theta-\theta^k\right\|_{\theta^k}^2\right),\nonumber
    \end{eqnarray}
where $M$ is a positive number. 
The exact description of the algorithm is provided below; you will recognize it as identical to CRNAS modulo the altered subproblem. 

\vspace{0.10in}
\shadowbox{\begin{minipage}{6.4in}
{\bf First-Order Affine Scaling}\label{alg1}

\begin{description}
\item[Step 0:] Provide an interior point $\theta^0$, i.e., $A\theta^0=b$ and $\theta^0\in \mbox{\rm int}(\mathcal{K})$; choose the constants $\eta >0$, $M >0$, and $\alpha \in (0,1)$; set $k=0$

\item[Step 1:] For $k=0,1,...,K-1$, solve the following subproblem:
\begin{eqnarray}
        \theta^{k+1}&=&\arg \min _{\theta:A\theta=b,\|\theta-\theta^k\|_{\theta^k}\leq 1-\alpha} \left(\left\langle\nabla L\left(\theta^k\right), \theta-\theta^k\right\rangle +\frac{M}{2}\left\|\theta-\theta^k\right\|_{\theta^k}^2\right)\nonumber
    \end{eqnarray}

\item[Step 2:] If $\|\theta^{k+1}-\theta^k\|_{\theta^k}<\eta$, let $K=k+1$ and stop. Otherwise, go back to Step 1 with $k=k+1$
\end{description}
\end{minipage}}
\vspace{0.01in}

As CRNAS was well-defined, so is FOAS due to Lemma \ref{lemma:feas}, 
which guarantees all of the generated iterates lie inside the interior of the cone. 
So, we proceed to prove the convergence rate of FOAS. 
To begin, we assume the Lipschitz smoothness of the gradient.

\begin{assumption}\label{lip1} 
There exists a constant $\beta \geq 0$ such that 
for all $x, y \in \mbox{\rm int}(\mathcal{K})$ 
\[
\|\nabla L(y)-\nabla L(x)\|_x^*\leq \beta \|y-x\|_x.
\]
\end{assumption}

%
Using this assumption it follows there exists a local quadratic upper bound for $L$ at all points inside the interior of the cone. 

\begin{lemma}\label{taylor1}
Under Assumption \ref{lip1}, the following inequality holds for all $x, y \in \mbox{\rm int}(\mathcal{K})$
\[
L(y) \leq L(x)+\nabla L(x)^\top(y-x)+\frac{\beta}{2}  \|y-x\|_x^2.
\]
\end{lemma}

\begin{proof}
According to the fundamental theorem of calculus and norm relations we have
    \begin{eqnarray}
        L(y) - L(x)-\nabla L(x)^\top(y-x)
        &= & \left\langle\int_{0}^{1} (\nabla L(x+t(y-x))-\nabla L(x)) dt,y-x\right\rangle\nonumber\\
        &\leq&\left\|\int_{0}^{1} (\nabla L(x+t(y-x))-\nabla L(x)) dt\right\|_{x}^*\|y-x\|_x\nonumber\\
        &\leq&\left(\int_{0}^{1}\left\|(\nabla L(x+t(y-x))-\nabla L(x)) \right\|_{x}^*dt\right)\|y-x\|_x\nonumber\\
        &\leq&\left(\int_{0}^{1} \beta\left\|t(y-x)\right\|_{x}dt\right)\|y-x\|_x\nonumber\\
        &=& \frac{\beta}{2}  \|y-x\|_x^2\nonumber
    \end{eqnarray}
    where the last inequality follows from Assumption \ref{lip1}.
\end{proof}

Lemma~\ref{taylor1} demonstrates that, provided $M$ is large enough, FOAS at each iteration minimizes a local quadratic upper bound of our objective function. 
We show the function value decreases can be 
lowered bounded by the norm of $\nabla L(\theta^{k+1})$ projected in the linear space in the following two lemmas.

\begin{lemma}\label{decrease1}
    If $M\geq 2\beta$, then for all the iterates generated by FOAS we have that
    \[L(\theta^k)-L(\theta^{k+1})\geq \frac{M}{4}\|\theta^{k+1}-\theta^k\|_{\theta^k}^2.\]
\end{lemma}
\begin{proof}
Since $\theta^{k+1}$ is the optimal solution of the FOAS subproblem, we have
    \begin{eqnarray}
    L(\theta^k)&\geq &L(\theta^k)+\left\langle \nabla L(\theta^k), \theta^{k+1}-\theta^k\right\rangle+\frac{M}{2}\|\theta^{k+1}-\theta^k\|_{\theta^k}^2\nonumber\\
    &\geq &L(\theta^{k+1})+\frac{M}{4}\|\theta^{k+1}-\theta^k\|_{\theta^k}^2,\nonumber
    \end{eqnarray}
where the last step follows from Lemma \ref{taylor1}.
\end{proof}

Because of the linear constraint $A\theta=b$, we expect $(\nabla L(\theta^{k}))^\top d$ to converge to zero after affine scaling, where $d$ is any given feasible direction.

\begin{lemma} \label{lemma11}
    Assume the constraint $\|P\theta-\theta^k\|_{\theta^k}\leq 1-\alpha$ in the subproblem is inactive, namely $\|\theta^{k+1}-\theta^k\|_{\theta^k}< 1-\alpha$. Then, for any $d$ such as $Ad=0$ we have
    \[
    |\left(\nabla L(\theta^k)\right)^\top d| \leq M\|\theta^{k+1}-\theta^k\|_{\theta^k}\|d\|_{\theta^k}.
    \]
\end{lemma}
\begin{proof}
Because the constraint $\|\theta-\theta^k\|_{\theta^k}\leq 1-\alpha$ is inactive, any $d$ satisfying $Ad=0$ must be a feasible direction. According to the optimality condition, we have
\[\left(\nabla L(\theta^{k})+M(\nabla^2 B(\theta^k))(\theta^{k+1}-\theta^k)\right)^\top d=0,\]
which implies
\[
|\left(\nabla L(\theta^k)\right)^\top d| = M|(\theta^{k+1}-\theta^k)^\top (\nabla^2 B(\theta^k))d|\leq M\|\theta^{k+1}-\theta^k\|_{\theta^k}\|d\|_{\theta^k}.
\]
\end{proof}

Given the fact
$\sum_{k=0}^{K-1} (L(\theta^k)-L(\theta^{k+1}))=L(\theta_0)-L(\theta^K)\leq L(\theta_0)-L^*$, there exists a $0\leq k\leq K-1$, such that $L(\theta^k)-L(\theta^{k+1})\leq \frac{L(\theta_0)-L^*}{K} $.

Therefore, with Lemma \ref{decrease1}, we obtain
\[\frac{L(\theta^0)-L^{*}}{K}\geq L(\theta^k)-L(\theta^{k+1})\geq \frac{M}{4}\|\theta^{k+1}-\theta^k\|_{\theta^k}^2,\]
which implies

\[
\|\theta^{k+1}-\theta^k\|_{\theta^k}^2\leq \frac{4(L(\theta^0)-L^{*})}{KM}.
\]

Therefore, as long as $K> \frac{4(L(\theta^0)-L^{*})}{M(1-\alpha)^2}$ and $M\geq 2\beta$, according to Lemma \ref{lemma11}, we have the following bounds.

For any $d$, such as $Ad=0$, we have
\[|(\nabla L(\theta^{k}))^\top  d|\leq M\|\theta^{k+1}-\theta^k\|_{\theta^k}\|d\|_{\theta^k}\leq
2(L(\theta^0)-L^{*})^{1/2}M^{1/2}K^{-1/2}\|d\|_{\theta^k}. \]

We arrive at the following theorem.
\begin{theorem}
    Use First-Order Affine Scaling Algorithm \ref{alg1} and assume Assumption \ref{lip1} and $M>2\beta$. For any $K> \frac{4(L(\theta^0)-L^{*})}{M(1-\alpha)^2}$, there is a $0\leq k\leq K-1$, such that for any $d$ with $Ad=0$, we have
    \[|(\nabla L(\theta^{k}))^\top  d|\leq 2(L(\theta^0)-L^{*})^{1/2}M^{1/2}K^{-1/2}\|d\|_{\theta^k}.\]
\end{theorem}

\subsection{Deterministic Drug-Affected Cell Proliferation Experiment Details}\label{appx: deterministic initialization}

To set up our optimization problems, we collected \textit{in silico} data according to a true parameter set $\theta_{PP}^{*}(S)$ and the deterministic framework at collections of time points $\mathcal{T}$ and drug dose levels $\mathcal{D}$. 
Without further specification, we employ the following time points and drug dose levels: 
\begin{equation}
    \label{eq: Time and dose levels}
    \begin{split}
        \mathcal{T} &= \{0,3,6,9,12,15,18,21,24,27,30,33,36\}\\
        \mathcal{D} &= \{0,0.0313,0.0625,0.125,0.25,0.375,0.5,1.25,2.5,3.75,5\}.
    \end{split}
\end{equation}
Since we assume no noise in these experiments, we only collect data for one replicate. 
We selected the true parameter set $\theta_{PP}^{*}(S)$ randomly from a biologically feasible range, denoted as $\Theta^*(S)$. 
This range is specified for each experiment in the subsequent sections. 
With the true parameter set, we obtained the cell count data of each sub-type and the following total cell count data according to \eqref{eq:subpopulation dynamic}.

Based on the data, we solved \eqref{eq: PhenoPop Optimization} to obtain the point estimation $\hat{\theta}_{PP}(S)$ within the optimization feasible range $\hat{\Theta}(S)$. 
The optimization feasible range is larger than the biologically feasible region, indicating incomplete prior knowledge about the `true nature' of the parameters.\\


\noindent\underline{Experiment with $S = 1$}\\

The biologically feasible range and optimization feasible range for this experiment are described in Table \ref{tab: Deterministic range S1}. Note that the true EC50 parameter $E$ should be located within the drug dose levels $\mathcal{D}$ to ensure parameter identifiability.
\begin{table}[ht]
    \centering
    \begin{tabular}{|c|c|c|c|c|}
    \hline
      & $\alpha$  & $b$ & $E$ & $n$ \\
    \hline
      $\Theta^*(1)$ & $(0,0.1)$  & $(0.8,1)$ & $(0.05,0.1)$ & $(1.5,5)$\\
       \hline
      $\hat{\Theta}(1)$ & $(0,1)$ & $(0,1)$ & $(0,\infty)$ & $(0,\infty)$\\
      \hline
    \end{tabular}
    \caption{Biologically feasible range and optimization feasible range for the deterministic framework experiment with $S = 1$}
    \label{tab: Deterministic range S1}
\end{table}

\noindent\underline{Experiment with $S = 2$}\\

When $S = 2$, we denote the subpopulation with smaller EC50 value as `sensitive' and the subpopulation with a higher EC50 value as `resistant'. 
The sensitive subpopulation EC50 parameter $E_s$ is selected to be distinct from $E_r$ of the resistant subpopulation for model identifiability. 
Other than the EC50 values, we do not distinguish the parameter space between the sensitive and resistant subpopulations. 
We also note the initial proportions $p_s$ and $p_r$ are selected to satisfy $p_s + p_r = 1$. 
Table \ref{tab: Deterministic range S2} outlines the biologically feasible range and the feasible region used in the optimization model for this experiment.\\

\begin{table}[ht]
    \centering
    \begin{tabular}{|c|c|c|c|c|c|c|}
    \hline
      & $p$ & $\alpha$  & $b$ & $E_s$ & $E_r$ & $n$  \\
    \hline
     $\Theta^*(2)$ &  $(0,1)$ &$(0,0.1)$  & $(0.8,1)$ & $(0.05,0.1)$ & $(0.5,2.5)$ & $(1.5,5)$  \\
    \hline
     $\hat{\Theta}(2)$ &  $(0,1)$ &$(0,1)$  & $(0,1)$ & $(0,\infty)$ & $(0,\infty)$ & $(0,\infty)$  \\
     \hline
    \end{tabular}
    \caption{Biologically feasible range and optimization feasible range for the deterministic framework experiment with $S = 2$}
    \label{tab: Deterministic range S2}
\end{table}

\noindent\underline{Experiment with $S > 2$}\\

In this experiment, we dynamically selected the drug dose levels. 
Specifically, we selected the EC50 biologically feasible range and dose levels based on the number of subpopulations $S$. 
We kept the overall range of $\mathcal{D}$ the same, i.e. $\min(\mathcal{D}) = 0, \max(\mathcal{D}) = 10$, 
and selected $4S$ dose levels within this range according to a logarithmic scale. 
For example, with $S=3$ the drug doses could be selected as
\begin{align*}
    \mathcal{D}(3) = \{0,0.01,0.02,0.0398,0.0794,0.1585,0.3162,0.631,1.2589,2.5119,5.0119,10\}.
\end{align*}

To preserve parameter identifiability, we designed the biologically feasible ranges for the EC50 parameters based on the generated dose levels. 
For instance, we have the following biologically feasible ranges for the experiment $S = 3$ described in Table \ref{tab: Deterministic range S3}, where $E_1,E_2,E_3$ are from distinct subpopulations and the ranges of the other parameters are the same across the subpopulations.

\begin{table}[ht]
    \centering
    \begin{tabular}{|c|c|c|c|c|c|c|c|}
    \hline
      & $p$ & $\alpha$  & $b$ & $E_1$ & $E_2$ & $E_3$ & $n$  \\
    \hline
     $\Theta^*(3)$ &  $(0,1)$ &$(0,0.1)$  & $(0.8,1)$ & $(0.005,0.0299)$ & $(0.119,0.4736)$ & $(1.8854,7.5059)$ & $(1.5,5)$  \\
    \hline
     $\hat{\Theta}(3)$ &  $(0,1)$ &$(0,1)$  & $(0,1)$ & $(0,\infty)$ & $(0,\infty)$ & $(0,\infty)$ & $(0,\infty)$  \\
     \hline
    \end{tabular}
    \caption{Biologically feasible range and optimization feasible range for the deterministic framework experiment with $S = 3$}
    \label{tab: Deterministic range S3}
\end{table}

\subsection{Stochastic Drug-Affected Cell Proliferation Experiment
Details}\label{appx: stochastic initialization}

Similar to the deterministic framework experiment, we randomly generated $100$ true parameter sets $\theta_{LBD}^{*}(S)$ from biologically feasible ranges and the corresponding simulated dataset. 
Different from the PhenoPop experiment, we assumed the stochastic linear birth-death process governed the tumor dynamic. 
Consequently, we employed the Gillespie algorithm \cite{gillespie1976general} to simulate the data according to a stochastic process. 
Due to the stochastic nature of the simulation, we collected 13 replicates under the same drug dose levels and time points as in \eqref{eq: Time and dose levels}. 
The biologically feasible ranges and the constraints for the optimization problems are presented in Table \ref{tab: Stochastic range S2}.

\begin{table}[ht]
    \centering
    \begin{tabular}{|c|c|c|c|c|c|c|c|}
    \hline
      &$p$ & $\beta$ & $\nu$  & $b$ & $E_s$ & $E_r$ & $n$  \\
    \hline
    $\Theta^*$ & $(0,1)$ &$(\nu,\nu+0.1)$ & $(0,1)$  & $(0.8,1)$ & $(0.05,0.1)$ & $(0.5,2.5)$ & $(1.5,5)$  \\
    \hline
    $\hat{\Theta}$ & $(0,1)$  &$(0,1)$ & $(0,1)$  & $(0,1)$ & $(0,\infty)$  & $(0,\infty)$ & $(0,\infty)$\\
    \hline
    \end{tabular}
    \caption{Biologically feasible range and optimization feasible range for the stochastic framework experiment with $S = 2$}
    \label{tab: Stochastic range S2}
\end{table}

\subsection{Heterogeneous Logistic Model Experiment Details}\label{appx: heterogeneous logistic model details}

Similar to the experiments in Section \ref{sec: Case study}, we examined the algorithms' performance for 100 independent experiments with distinct true parameter sets $\theta^*_{LG}(S)$. 
The true parameter sets were randomly selected from the biologically feasible range shown in Table \ref{tab: heterogeneous logistic S2}. 
With the true parameter set, we generated the experimental data at time points 
\[\mathcal{T} = \{0,1.111,2.222,3.333,4.444,5.555,6.666,7.777,8.888,10\},\]
according to the relation described in \eqref{eq: heterogeneous logistic model}. 
The optimization problems were then formulated using the feasible regions described in Table \ref{tab: heterogeneous logistic S2} as in the prior experiments.

\begin{table}[ht]
    \centering
    \begin{tabular}{|c|c|c|c|c|c|c|}
    \hline
      &$p_1$ & $\alpha_1$ & $\beta_1$ & $p_2$ & $\alpha_2$ & $\beta_2$  \\
    \hline
    $\Theta^*$ & $(0,1)$ &$(0,1)$ &  $(0,1)$  & $1-p_1$ & $(2,3)$ & $(2,3)$  \\
    \hline
    $\hat{\Theta}$ & $(0,1)$  &$(0,10)$ & $(0,10)$  & $(0,1)$ & $(0,10)$  & $(0,10)$.\\
    \hline
    \end{tabular}
    \caption{Biologically feasible range and optimization feasible range for the heterogeneous logistic growth model with $S = 2$}
    \label{tab: heterogeneous logistic S2}
\end{table}

\subsection{CRNAS Stopping Criteria}\label{appx: stopping criteria}

To ensure the fairness of comparison in the computational time of each optimization algorithm, we employ two stopping criteria for CRNAS, which are also implemented in \textit{fmincon}. One is the first-order optimality condition, and the other is the stepsize of each iteration. In particular, the algorithm will stop once one of the following quantities falls below the given threshold $\epsilon = 10^{-6}$:
\[\begin{cases}
    \|\nabla f(\theta_{k+1}) \| & \text{ Euclidean norm of the gradient, }\\
    \| \theta_{k+1} - \theta_{k}\| & \text{ Euclidean norm of the difference in each iteration solution,}
\end{cases}
\]
where $f(\theta)$ is the objective function. It is worth noticing that \textit{fmincon} does not use this measurement for the first-order optimality condition. However, both methods of computing the first-order optimality conditions are zero at the minimum. Therefore, we use the same threshold for both CRNAS and the \textit{fmincon} algorithms. In addition to these two stopping criteria, the algorithm will also stop if the number of iterations is greater than $500$.

\bibliographystyle{alpha}
\bibliography{sample}

\end{document}